\documentclass[reqno,centertags,12pt]{amsart}
\usepackage{amsmath,amsthm,amscd,amssymb}
\usepackage{latexsym}
\usepackage{url}
\usepackage{pstricks}
\usepackage{pst-node}
\usepackage {verbatim}

\newcommand{\Hmm}[1]{\leavevmode{\marginpar{\tiny%
$\hbox to 0mm{\hspace*{-0.5mm}$\leftarrow$\hss}%
\vcenter{\vrule depth 0.1mm height 0.1mm width \the\marginparwidth}%
\hbox to 0mm{\hss$\rightarrow$\hspace*{-0.5mm}}$\\\relax\raggedright #1}}}

\newcommand{\R}{{\mathbb{R}}}

\newcommand{\ol}{\overline}

\newcommand{\wti}{\widetilde  }

\newcommand{\hatt}{\widehat}
\newcommand{\beq}{\begin{equation}}
\newcommand{\eeq}{\end{equation}}
\newcommand{\bdm}{\begin{displaymath}}
\newcommand{\edm}{\end{displaymath}}
\newcommand{\ba}{\begin{align}}
\newcommand{\ea}{\end{align}}
\newcommand{\bpf}{\begin{proof}}
\newcommand{\epf}{\end{proof}}

\newcommand{\la}{\langle}
\newcommand{\ra}{\rangle}

\newcommand{\essinf}{\mathop{\text{essinf}}}

\newcommand{\supp}{\mathrm{supp}\, }               
\newcommand{\dist}{\mathrm{dist}}               

\newcommand{\e}{\mathrm{e}}
\renewcommand{\d}{\mathrm{d}}

\newcommand{\veps}{\varepsilon}

\newcommand{\re}{\mathrm{Re}}

\newcommand{\abs}[1]{\lvert#1\rvert}

\newcommand{\id}{\mathbf{1}}                

\DeclareMathOperator{\sgn}{sgn}

\allowdisplaybreaks

\newcommand{\calC}{\mathcal{C}}

\newcommand{\calQ}{\mathcal{Q}}

\newcommand{\calS}{{\mathcal{S}}}

\newtheorem{thm}{Theorem}
\newtheorem{prop}[thm]{Proposition}
\newtheorem{lem}[thm]{Lemma}
\newtheorem{cor}[thm]{Corollary}

\theoremstyle{definition}

\newtheorem{definition}[thm]{Definition}

\newtheorem{rem}[thm]{Remark}
\newtheorem{rems}[thm]{Remarks}
\newtheorem{ass}[thm]{Assumption}

\newcounter{theoremi}[thm]

\newcommand{\itemthm}{\refstepcounter{theoremi} {\rm(\roman{theoremi})}{~}}

\numberwithin{thm}{section}
\numberwithin{equation}{section}

\newcounter{smalllist}

\newcounter{smallenum}

\begin{document}

\title[Exponential decay of dispersion managed solitons]{Exponential decay of dispersion managed solitons for general dispersion profiles}

\author{William~R.~Green and Dirk Hundertmark}
\address{Department of Mathematics \\
Rose-Hulman Institute of Technology \\
5500 Wabash Ave. \\
Terre Haute, IN 47803 USA  }
\email{green@rose-hulman.edu }
\address{Institute for Analysis, Department of Mathematics,  Karlsruhe Institute of Technology, 76128 Karlsruhe, Germany \\
and Department of Mathematics, Altgeld Hall
    and Institute for Condensed Matter Theory at the
    University of Illinois at Urbana--Champaign,
    1409 W.~Green Street, Urbana, IL 61801 USA }%
\email{dirk.hundertmark@kit.edu}%

\thanks{\copyright 2013 by the authors. Faithful reproduction of this article,
        in its entirety, by any means is permitted for non-commercial purposes}
\thanks{Supported in part by the Alfried Krupp von Bohlen und Halbach Foundation and 
NSF--grant DMS-0803120 (D.~H.) and NSF--grant DMS-0900865 and an AMS--Simons Travel Grant (W.R.G.).}

\begin{abstract}
We show that any weak solution of the dispersion management equation describing dispersion managed solitons 
together with its Fourier transform decay exponentially. This strong regularity result extends a recent result of 
Erdo\smash{\u{g}}an, Hundertmark, and Lee in two directions, to arbitrary non-negative average dispersion and, 
more importantly, to rather general dispersion profiles, which cover most, if not all, physically relevant cases.
\end{abstract}

\maketitle

\section{Introduction}
\label{introduction}


We prove strong regularity properties of dispersion management solitons, which are stationary solutions of the 
dispersion management, or often also called Gabitov-Turitsyn, equation
 \beq\label{eq:DMS}
   i\partial_t u = -d_{\text{av}} \partial_x^2 u - Q(u,u,u),
 \eeq
where the non-linear term is a cubic nonlinearity which is also highly non-local, given by
 \beq\label{eq:Q}
  Q(f_1,f_2,f_3) := \frac{1}{L} \int_0^L T^{-1}_{D(s)}\big[ (T_{D(s)}f_1)(\ol{T_{D(s)}f_2})(T_{D(s)}f_3)  \big]\, ds .
 \eeq
Here $T_r:= e^{ir\partial_x^2}$ is the solution operator of the free Schr\"odinger equation and $D(s):= \int_0^s d_0(\alpha)\,d\alpha$. The dispersion profile $d_0$ is an integrable periodic function with period $L$ and mean zero. Physically, it is the mean zero part of the local dispersion in dispersion managed glass--fiber cables.
Equation  \eqref{eq:DMS}  describes the (average) profile of signals in dispersion managed glass fiber cables in a frame moving with the group velocity of the pulse. Thus stationary solutions of \eqref{eq:DMS}, specifically solutions of \eqref{eq:DMS-stationary} and \eqref{eq:DMS-weak} below, 
lead to (quasi-)periodic breather type solitary pulses in these fiber cables, which are important in the development of ultra--fast long--haul optical data 
transmission fibers.
For the convenience of the reader, 
we first discuss the reason for the importance of equation \eqref{eq:DMS} and its properties for non-linear fiber optics before stating our main results.

Optical pulses in glass finer cables are very well described by the non-linear Schr\"odinger equation
 \beq\label{eq:NLS}
      i \partial_t u = -d \partial_x^2 u - |u|^2 u,
 \eeq
where the $d$ is the (local) dispersion of the glass fiber. One might argue that light pulses in glass--fiber cables are, of course, described by Maxwell's equations. However, in the slowly varying envelope approximation in non-linear optics \eqref{eq:NLS} indeed describes the evolution of the envelope of a pulse in a frame moving with the group velocity of the signal through a glass fiber cable, see \cite{Mitschke,MN03,SulemSulem}.
As a \emph{warning}: with our choice of notation the variable $t$ denotes the position along the glass fiber cable  and $x$ the (retarded) time. 

The possibility to periodically manage the dispersion by putting alternating sections with
positive and negative dispersion together in an optical glass-fiber cable to compensate for dispersion of the signal was predicted by Lin, Kogelnik, and Cohen already in 1980, see \cite{LKC80}, and then implemented by Chraplyvy and Tkach for which they received the Marconi prize in 2009, see \cite{Marconi}. 
In this case the local dispersion $d(t)$, recall that $t$ is the distance along the cable, is periodically varying along the cable and the non-linear Schr\"odinger equation \eqref{eq:NLS} changes periodically from focussing to defocussing, depending on the sign of the local dispersion $d(t)$. This periodically varying dispersion creates a new optical fiber type, e.g., through a judicious choice of parameters both the so called Gordon Haus effect and the four wave mixing can be suppressed, enabling the development of long--haul optical fiber transmission systems with record breaking capacities beyond one Terabit/second per fiber which equates to a more than 100-fold capacity increase in the last ten years. 
Thus dispersion management technology has been of fundamental importance for ultra-high speed data transfer through glass fiber cables over intercontinental distances and is now widely used commercially, \cite{AB98, CGTD93, Chraplyvy-etal95, GT96a, GT96b, KH97, Kurtzge93, LK98, LKC80, MM99, MGLWXK03, MMGNMG-NV99}.
For a review of the physics of dispersion management see \cite{TDNMSF99, Turitsynetal03}.

The periodic modulation of the dispersion can be described by the ansatz
 \beq\label{eq:dispersion}
    d(t) = \veps^{-1} d_0(t/\veps) + d_{av} .
 \eeq
Here  $d_{av} \ge 0$ is the average component and $d_0$ its mean zero part which we assume to have period $L$.
For small $\veps$ the equation \eqref{eq:dispersion} describes a fast strongly varying dispersion, which corresponds to  the regime of \emph{strong} dispersion management.

Let  $D(t)= \int_{0}^t d_0(s)\, ds$ and note that as long as  $d_0$ is locally
integrable and has period $L$ with mean zero, $D$ is also periodic with period $L$.
Furthermore,  $T_r= e^{ir\partial_x^2}$ is a unitary operator and thus the
unitary family $t\mapsto T_{D(t/\veps)}$ is periodic with period $\veps L$.
Making the ansatz $u(t,x)= (T_{D(t/\veps)}v(t,\cdot))(x)$ in \eqref{eq:NLS}, a short calculation
yields
    \beq\label{eq:NLS-2}
        i\partial_t v= d_{\text{av}}\partial_x^2 v - T_{D(t/\veps)}^{-1}\big[|T_{D(t/\veps)} v|^2 T_{D(t/\veps)} v\big]
    \eeq
which is equivalent to \eqref{eq:NLS} and still a non-autonomous equation.

For small $\veps$, that is, in the regime of strong dispersion management,
$T_{D(t/\veps)}$ is fast oscillating in the variable $t$, hence the
solution $v$ should evolve on two widely separated time-scales, a slowly evolving part
$v_{\text{slow}}$ and a fast oscillating part, which is hopefully small.
Analogously to Kapitza's treatment of the unstable pendulum which is stabilized by fast
oscillations of the pivot, see \cite{LandauLifshitz}, the effective equation for the slow part
$v_{\text{slow}}$ was derived by Gabitov and Turitsyn \cite{GT96a,GT96b} and a bit later by
Ablowitz and Biondini \cite{AB98}. It is given by integrating the fast
oscillating term containing $U_{D(t/\veps)}$ over one period in $t$, treating the slowly varying part $v_{\text{slow}}$ as constant in $t$ over the small time--scale period $\veps L$,
    \beq
    \begin{split}\label{eq:GT-time-version-1}
        i\partial_t v_{\text{slow}}
        &=  d_{\text{av}}\partial_x^2 v_{\text{slow}}
            - \frac{1}{\veps L}\int_{0}^{\veps L}
                T_{D(s/\veps)}^{-1}\big[|T_{D(s/\veps)} v_{\text{slow}}|^2 T_{D(s/\veps)} v_{\text{slow}}\big]
              \, ds \\
        &= d_{\text{av}}\partial_x^2 v_{\text{slow}}
        - Q(v_{\text{slow}},v_{\text{slow}},v_{\text{slow}}) ,
    \end{split}
    \eeq
where we made a simple change of variables $s/\veps\to s$ and used \eqref{eq:Q}. Of course, \eqref{eq:GT-time-version-1} is nothing but \eqref{eq:DMS} for the slow part $v_{\text{slow}}$
of the pulse in a dispersion managed glass--fiber cable.

This averaging procedure leading to \eqref{eq:GT-time-version-1} was rigorously justified
for certain dispersion profiles in \cite{ZGJT01}:
Given an initial condition $f$ in a suitable Sobolev space, the solutions of \eqref{eq:GT-time-version-1} and
\eqref{eq:NLS-2} stay $\veps$ close -- measured in certain Sobolev norms -- over long distances
$0\le t\le C/\veps$, see \cite{ZGJT01} for the precise formulation.
Thus of special interest are stationary solutions of \eqref{eq:DMS-stationary} or its weak formulation \eqref{eq:DMS-weak}, since they lead to breather like (quasi-)periodic solutions for the original equation
\eqref{eq:NLS}, whose average profile, for long $t\le C \veps^{-1}$, does not change (much) even for large $t$, i.e., long cables, which cover, for example, intercontinental distances.

Making the ansatz $u(t)= e^{i\omega t}f$ in \eqref{eq:DMS} yields the time independent Gabitov-Turitsyn equation
 \beq\label{eq:DMS-stationary}
   -\omega f = -d_{\text{av}} \partial_x^2 f - Q(f,f,f).
 \eeq
There is a corresponding weak formulation. Taking the scalar product with $g\in \calC^\infty_0(\R)$ or $H^1(\R)$ in \eqref{eq:DMS-stationary} and doing a judicious integration by parts using that $T_r$ is unitary gives
 \beq\label{eq:DMS-weak}
   -\omega \la g, f\ra = \d_{\text{av}}\la g',f' \ra -\calQ(g,f,f,f,)
 \eeq
for all $g\in H^1(\R)$, where the four-linear functional $\calQ$ is, at least informally, given by
 \beq\label{eq:calQ}
   \calQ(f_1,f_2,f_3,f_4):=
     \frac{1}{L} \int_0^L \int \ol{T_{D(s)}f_1(x)}T_{D(s)}f_2(x)\ol{T_{D(s)}f_3(x)}T_{D(s)}f_4(x)\, ds dx.
 \eeq
Of course, this is, at the moment only on an informal level, since it is not a-priori clear that $\calQ$ or, equivalently, $Q$ are well-defined if $f_l\in L^2(\R)$, or even if $f_l\in H^1(\R)$. However, they turn out to be well defined for a large class of dispersion profiles $d_0$, see Remark \ref{rem:boundedness} below. In fact, it turns out that $Q$ and $\calQ$ are well--defined as soon as $d_0$ changes sign only finitely many times over its period and $d_0^{-1}\in L^1([0,L])$, see \cite{HLexistence}.

Due to its importance for ultra-high speed long--haul data transfer through glass fiber cables, there is an enormous amount of work on the level of numerics and theoretical physics on the dispersion management technique, see, for example,  \cite{TBF12,TDNMSF99,Turitsynetal03} and the references therein, but rigorous results are rare, we know of \cite{EHL,HL09,HLexistence,Kunze04,KMZ05,Stanislavova05,ZGJT01}.
This is mainly caused by the highly non--local and oscillatory nature of the nonlinearity $Q$, which makes it hard to study.
Moreover, except for \cite{HLexistence,ZGJT01} these rigorous results consider only a model case for the dispersion profile $d_0$, which is the simplest choice of  
of a periodic mean--zero dispersion profile given by
 \beq\label{eq:d0-simplest}
   d_0 = \id_{[0,1)} - \id_{[1,2)} ,
 \eeq
 where $\id_A$ is the characteristic function of the set $A$. This models the simplest case of a dispersion managed cable; made up of alternating two different pieces of glass fibers whose dispersions compensate each other. In this case $D:[0,2]\to\R$ is simply the tent function $D(s)= (1-|s-1|)_+=\max(0,1-|s-1|)$. 

In addition, except for \cite{ZGJT01}, the rigorous works mentioned above consider only the case of vanishing average dispersion and  most of them address the existence problem but do not study much the fine properties of these solitons. 
For the model case \eqref{eq:d0-simplest}, Lushnikov gave convincing, but non--rigorous, heuristic arguments that the dispersion managed soliton solutions of \eqref{eq:DMS-stationary} should decay exponentially \cite{Lushnikov04}. He conjectured that for $d_\text{av}=0$ the asymptotic form of the dispersion--managed solitons should be given by
 \beq\label{eq:Lushnikov}
  f(x)\sim A\cos(a_0 x^2 + a_1 x + a_2) e^{-b|x|} \quad \text{~as~} x\to\infty
 \eeq
for some constants $A,a_j,$ $j=0,1,2$, and $b>0$. Moreover, he confirmed this conjecture numerically and also argued that the exponential decay of dispersion managed solitons should also hold for positive average dispersion $d_{\text{av}}>0$.

The motivation for our work comes from the fact that recently a rigorous result became available in which supports in part Lushnikov's conjecture. It was shown in \cite{EHL} that dispersion management solitons decay exponentially both in time and frequency domains, at least in the model case \eqref{eq:d0-simplest} and for vanishing average dispersion. In the present work we extend the results from \cite{EHL} to arbitrary non--negative average dispersion $d_{\mathrm{av}}$ and to a large class of dispersion profiles $d_0$ which, we believe, covers most, if not all, physically relevant cases:

\begin{ass}[Assumptions on $d_0$] \label{def:d0conditions}

a) Given $L>0$, we say that the dispersion profile $d_0$ is piecewise $W^{1,1}$ on $[0,L]$
   (or just piecewise $W^{1,1}$) if there is a finite partition $0=c_0<c_1<\ldots <c_N=L$ with $N<\infty$ such that on each interval $I_j:=(c_j,c_{j+1})$ the dispersion profile $d_0$ is integrable with an integrable weak derivative.\\[0.2em]
 b) Furthermore, $d_0:[0,L]\to\R$ is bounded away from zero if
 \beq\label{eq:bounded-away-zero}
   \essinf_{s\in[0,L]} |d_0(s)| > 0 .
 \eeq
\end{ass}
\begin{rem}
Conditions a) and b) above are certainly natural. The cables in a dispersion managed fiber are made up by concatenating periodically \emph{finitely many} different pieces of glass--fibers which have different dispersion indices. One could argue that within such a piece the dispersion usually is, to a good approximation, even constant. The resulting diffraction profile $d_0$ would thus even be \emph{piecewise constant} and certainly bounded away from zero. 
However, due to unavoidable imperfections in the manufacturing process of those glass--fiber cables, it is unavoidable that the dispersion within a piece will locally vary, at least by a small amount, and it would be rather unfortunate if the mathematical theory would necessarily require a piecewise constant dispersion profile.
 With our method, we can even allow for a smoothly varying local dispersion within each piece of glass--fiber cable, as long as it stays away from zero. More precisely, we only need that $d_0$ is absolutely continuous, i.e., has an integrable weak derivative, in each piece of the partition, but $d_0$ can jump and change sign from one piece of the partition to another as long as it does not approach zero. This should cover all physically relevant cases.
\end{rem}
The following theorem is the main result.
\begin{thm}[Exponential decay]\label{thm:exp-decay} Assume that the dispersion profile $d_0$ is
   piecewise $W^{1,1}$ and bounded away from zero on $[0,L]$ and let $f\in H^1$ be a solution of weak form of the \eqref{eq:DMS-weak} for some $\omega>0$ and $d_{\mathrm{av}}> 0$ {\rm (}respectively, $f\in L^2$ is a solution of \eqref{eq:DMS-weak} for $\d_\mathrm{av}=0${\rm )}. Then there there exist constants  $\mu,C>0$ such that for all $x,k\in\R$
		\bdm
			|f(x)|\le C e^{-\mu | x |}, \text{ and }
			|\hatt{f}(k)|\le C e^{-\mu | k |},
		\edm
		where $\hat{f}$ is the Fourier transform of $f$.
 \end{thm}
An useful consequence of Theorem \ref{thm:exp-decay} is
 \begin{cor}
   Assume that the dispersion profile $d_0$ is
   piecewise $W^{1,1}$ and bounded away from zero on $[0,L]$ . Then any weak solution $f\in H^1(\R)$ of the dispersion management equation \eqref{eq:DMS-weak} for some $\omega>0$ and $d_{\mathrm{av}}> 0$ {\rm (}respectively, weak solution $f\in L^2$ of \eqref{eq:DMS-weak} for $\d_\mathrm{av}=0${\rm)}
   is analytic in a strip containing the real axis. The same holds for its Fourier transform $\hatt{f}$.
 \end{cor}
\bpf
Since
 \bdm
  f(x)= \frac{1}{\sqrt{2\pi}} \int_{-\infty}^\infty e^{-ikx} \hatt{f}(k)\, dk \quad
\textrm{and} \quad
  \hatt{f}(k) = \frac{1}{\sqrt{2\pi}} \int_{-\infty}^\infty e^{ikx} f(x)\, dx,
 \edm
 the exponential bounds from Theorem \ref{thm:exp-decay} immediately show that $f$ and $\hatt{f}$ can be analytically continued onto the strip $\{z=x+iy :\, x\in\R, |y|<\mu\}$.
\epf

This result and, indeed, some of the technical machinery
used in its proof, in particular, the use of exponential weighted oscillatory 
integrals and multi-linear bounds, mirrors the result in \cite{EHL} for the model case, where the dispersion 
profile is given by  \eqref{eq:d0-simplest}. 
Note, however, that the representation \eqref{eq:rep-real-space} from Proposition \ref{prop:rep} shows that we 
have to bound oscillatory  integrals with the kernel $D(s)^{-1}$, which is singular whenever 
$D(s)=\int_0^s d_0(\alpha)\, d\alpha=0$ and 
since $d_0$ changes sign $D$ will have possibly plenty of zeros whose location is not known directly. 
The model case \eqref{eq:d0-simplest} is much easier: $D$ is known explicitly and 
one can straightforwardly implement the change  of variables  $r=D(s)$, which drastically simplifies the structure of the 
four-linear functional $\calQ$ and thus also the proof of the necessary multi-linear bounds for $\calQ$. 
Much of the essential additional difficulty caused by the extension to more general dispersion profiles away from the model case 
\eqref{eq:d0-simplest} to virtually all physically relevant dispersion profiles $d_0$  described by 
Assumption \eqref{def:d0conditions} and the extension to $d_{av}>0$  is unfortunately 
technical and explained in Remark~\ref{rem:difficulty}  below.

The paper is organized as follows:
In Section~\ref{sec:Q} we discuss and establish 
several bounds for the multi-linear operator
$\mathcal Q$.  In Section~\ref{sec:multilinear} we
introduce the majorant functionals $K^1$ and $K^2$ and
deduce needed estimates for $\mathcal Q$ from multi-linear
bounds on $K^j$, $j=1,2$.  In Section~\ref{sec:expdecay}
we use the results we establish in Sections~\ref{sec:Q}
and \ref{sec:multilinear} to prove 
Theorem~\ref{thm:exp-decay}.  Finally, we include a
technical appendix in which we prove the multi-linear
bounds on the operators $K^j$.

\section{A-priori estimates on $\calQ$}
\label{sec:Q}
First we introduce some necessary notation. For $1\le p\le \infty$ we denote by $L^p=L^p(\R)$ the usual scale of $L^p$--space with norm given by $\|\cdot\|_p$. We will only need $L^p$ on $\R$.
If $p=2$, we will simply write $\|\cdot\|$ for $\|\cdot\|_2$. Furthermore, $H^s=H^s(\R)$ is the usual scale of $L^2$ based Sobolev spaces with norm $\|\cdot\|_{H^s}$. The space $L^2$ is a Hilbert space with scalar product given by $\la g, f \ra= \int_\R \ol{g(x)}f(x)\, dx$, i.e., our scalar product in anti-linear in the first component and linear in the second. With $\delta$ we denote the one--dimensional Dirac measure, or Dirac distribution, i.e., for a (measurable, i.e., Borel) set $A$ one has $\delta(A)=\int_A\delta(s)\, ds= \id_A(0)$. Here $\id_A$ is the characteristic function, or indicator function, of the set $A$, i.e., $\id_A(x)=1$ if $x\in A$ and $\id_A(x)=0$ if $x\not\in A$. With $|A|$ we denote the Lebesgue measure of a Borel subset $A\subset\R^d$. 
Lastly, we find it convenient to write $f \lesssim g$ for two functions $f$ and $g$ if there is a universal constant $C>0$ such that
$f\le C g$. 


We start with a representation of $\calQ$ which is a bit more complicated but similar to the one in \cite{EHL}.

\begin{prop}\label{prop:rep}
 Assume that the dispersion profile $d_0$ is piecewise $W^{1,1}$ on $[0,L]$ and bounded away from zero. Then for $f_1,f_2,f_3, f_4$ in the Schwarz class of rapidly decaying smooth functions one has
  \begin{align}
   \calQ (f_1,&f_2,f_3,f_4) = \nonumber \\
   &
    \frac{1}{4\pi L}\int_{\R^4} \int_0^L
      e^{-i\frac{a(\eta)}{4D(s)}} \frac{ds}{|D(s)|} \,
     \ol{f_1(\eta_1)} f_2(\eta_2) \ol{f_3(\eta_3)} f_4(\eta_4)\,\delta(\eta_1-\eta_2+\eta_3-\eta_4)d\eta
     \label{eq:rep-real-space}\\
  \intertext{and }
   \calQ(f_1,&f_2,f_3,f_4) =\nonumber \\
   &
     \frac{1}{2\pi L} \int_{\R^4} \int_0^L
     e^{iD(s)a(\eta)}\, ds \,
     \ol{\hatt{f_1}(\eta_1)} \hatt{f_2}(\eta_2) \ol{\hatt{f_3}(\eta_3)} \hatt{f_4}(\eta_4)\,\delta(\eta_1-\eta_2+\eta_3-\eta_4) d\eta
     \label{eq:rep-fourier-space}
 \end{align}
 where $a(\eta):= \eta_1^2-\eta_2^2+\eta_3^2-\eta_4^2$ and $D(s):= \int_0^s d_0(\alpha)\,d\alpha$.
\end{prop}

\begin{proof}
Assume that $f_j$ is in the Schwarz class, then all expressions below are a-priori well defined. Using the representation
 \beq
   T_r f_j(x) = e^{ir\partial_x^2} f(x)
    = \frac{1}{\sqrt{4\pi i r}} \int_\R e^{i\frac{(x-y)^2}{4r}} f_j(y)\, dy
 \eeq
for the free Schr\"odinger evolution in \eqref{eq:calQ} and collecting terms one sees
 \beq
  \begin{split}
   &\calQ(f_1,f_2,f_3,f_4) = \\
   &   \frac{1}{(4\pi)^2}  \frac{1}{L}\int_0^L \int_\R \int_{\R^4}
         \e^{-i\frac{a(\eta)}{4D(s)}} \frac{1}{|D(s)|^2} \ol{f_1(\eta_1)} f_2(\eta_2) \ol{f_3(\eta_3)} f_4(\eta_4)
         \e^{i\frac{x(\eta_1-\eta_2+\eta_3-\eta_4)}{2D(s)}} \, d\eta \, dx ds
  \end{split}
 \eeq
Exchanging the $\eta$ and $x$ integration, making the change of variables $x=2D(s) z$, and noting that, as distributions, the Dirac measure is given by $\delta(\alpha)= \tfrac{1}{2\pi}\int_\R e^{iz\alpha}\, dz$ we arrive at
 \beq
  \begin{split}
   &\calQ(f_1,f_2,f_3,f_4) = \\
   &   \frac{1}{4\pi L}  \int_0^L \int_{\R^4}
         \e^{-i\frac{a(\eta)}{4D(s)}} \frac{1}{|D(s)|}\, \ol{f_1(\eta_1)} f_2(\eta_2) \ol{f_3(\eta_3)} f_4(\eta_4)
         \delta(\eta_1-\eta_2+\eta_3-\eta_4) \, d\eta \, ds
  \end{split}
 \eeq
which, after exchanging the $\eta$ and $s$ integration, is \eqref{eq:rep-real-space}.

The representation \eqref{eq:rep-fourier-space} follows from a similar line of reasoning, using the Fourier--repre\-sentation of the free Schr\"odinger evolution, 
 \beq\label{eq:free-schroedinger-fourier},
   T_r f(x) = \frac{1}{\sqrt{2\pi}} \int_\R e^{-ix k } \e^{-ir k^2} \hatt{f}(k)\, dk.
 \eeq
\end{proof}

\begin{rem}
 The oscillatory integral
   \beq\label{eq:osc-integ-1}
    \frac{1}{L}\int_0^L e^{iD(s)a}\, ds
   \eeq
 in \eqref{eq:rep-fourier-space} is certainly well-defined as long as $d_0$ is integrable on $[0,L]$.
 One needs to be more careful for the oscillatory integral
   \beq\label{eq:osc-integ-2}
   \frac{1}{L}\int_0^L e^{-i\frac{a}{4D(s)}} \frac{ds}{|D(s)|}  .
   \eeq
 For example, note that even in the model case \eqref{eq:d0-simplest}, where $d_0$ is the difference of characteristic functions, the integral $L^{-1}\int_0^L |D(s)|^{-1}\, ds = \int_0^1 r^{-1} \, dr $ diverges.
 Thus \eqref{eq:osc-integ-2} exists only as an oscillatory integral; to get useful bounds for it one has to invoke cancelations when $D$ is small,  see, in particular, Lemma \ref{lem:intbyparts2}, Remark \ref{rem:oscillatory-integral}, and Corollary \ref{cor:small-D} below.
 We note that the resulting bounds are a bit more involved than those for \eqref{eq:osc-integ-1}.
\end{rem}
The following is a key ingredient in our bounds for the oscillatory integrals. It allows us to control the oscillatory integrals without explicit 
knowledge of the location of the zeros of $D$. 
\begin{lem}[Sublevel set estimate]\label{lem:sublevel}
Let $L>0$, $d_0:[0,L)\to \R$ be bounded away from zero change and sign finitely many times.
Then, for $D(s):= \int_0^s d_0(\alpha)\, d\alpha$,
 \beq\label{eq:sublevel1}
  |\{ 0<s<L: \abs{D(s)}< \alpha \}| \lesssim \min(\alpha,1)
 \eeq
 for all $\alpha>0$.
\end{lem}
\begin{rem}
 This sublevel bound should be compared with the usual sublevel estimate which is used in the proof on van der Corput's lemma in harmonic analysis. 
 There one usually assumes that some derivative of $D$ has a fixed sign, but does not need the assumption that $D$ change sign finitely many times, 
 see Proposition 2.6.7 in \cite{Grafakos-classical-Fourier-analysis}.
\end{rem}
\bpf
Certainly $|\{ 0<s<L: \abs{D(s)}< \alpha \}|\le L\lesssim 1$.
The assumptions imply that $D$ has finitely many zeros in $[0,L]$ and at each zero of
$D$, the slope of $D$ is bounded below by $\delta := \essinf_{0\le s\le L}\abs{d_0(s)}>0$.
Thus, for small enough $\alpha$, $\{ 0<s<L: \abs{D(s)}< \alpha \}$ consists of a
finite number of intervals where each interval is of length at most $\alpha/\delta$. This shows
$|\{ 0<s<L: \abs{D(s)}< \alpha \}|\lesssim \alpha$ for small $\alpha$ and hence
\eqref{eq:sublevel1} holds.
\epf

\begin{cor}\label{cor:sublevel}
Under the assumptions of Lemma \ref{lem:sublevel}, we have
 \beq\label{eq:sublevel2}
  \int_0^L \frac{1}{\sqrt{D(s)}}\, ds <\infty .
 \eeq
\end{cor}
\bpf
Note
 \beq
  \begin{split}
    \int_0^L \frac{1}{\sqrt{D(s)}}\, ds
   &=
    \frac{2}{3} \int_0^L \int_{\alpha>\abs{D(s)}} \alpha^{-3/2}\, d\alpha d s 
   = \frac{2}{3} \int_0^\infty \alpha^{-3/2} |\{ 0<s<L: \abs{D(s)}< \alpha \}|\, d\alpha \\
   &
    \lesssim
     \int_0^\infty \alpha^{-3/2} \min(\alpha,1)\, d\alpha
      <\infty
  \end{split}
 \eeq
 due to \eqref{eq:sublevel1} and since the integrand in the last integral is integrable for $\alpha$ close to $0$ and close to $\infty$.
\epf

\begin{lem}\label{lem:intbyparts1}
Assume that $d_0$ is piecewise $W^{1,1}$ on $[0,L]$ and bounded away from zero. Then, for all $a\in \R$,
 \beq\label{eq:intbyparts1}
  \left| \int_0^L e^{iaD(s)}\, ds\right|
  \lesssim
  \frac{1}{\max(|a|,1)}
 \eeq
where $D(s):= \int_0^s d_0(\alpha)\, d\alpha $.
\end{lem}
\bpf
Certainly, $\big| \int_0^L e^{iaD(s)}\, ds\big| \le L <\infty$. On the other hand, let $N<\infty$ be the number of intervals on which $d_0$ is absolutely continuous and
$0=c_0<c_1<\ldots<c_N=L$ the partition of $[0,L]$ such that $d_0$ is absolutely continuous on each
interval $I_k=(c_{k-1},c_k)$. Then
 \beq\label{eq:sum-1}
  \int_0^L e^{iaD(s)}\, ds = \sum_{j=1}^N  \int_{I_j} e^{iaD(s)}\, ds
 \eeq
and since $D'=d_0$
 \beq\label{eq:intbyparts-base}
  e^{iaD(s)} = \Big( e^{iaD(s)} \Big)' \frac{1}{ia d_0(s)} ,
 \eeq
an integration by parts shows
 \beq
 \begin{split}
  \int_{I_j} e^{iaD(s)}\, ds
   &=
    \frac{1}{ia}\int_{I_j} \big( e^{iaD(s)} \big)' \frac{1}{d_0(s)} \, ds =
    \frac{1}{ia}
    \Big\{\Big[\frac{e^{iaD(s)}}{d_0(s)}\Big]_{c_{j-1}}^{c_j}
    + \int_{I_j} e^{iaD(s)} \frac{{d_0}'(s)}{d_0(s)^2} \, ds \Big\} .
 \end{split}
 \eeq
Note that ${d_0}'$ exists as an $L^1$-function on each interval $I_j$.
Thus, since $\delta:= \essinf\limits_{0\le s\le L}|d_0(s)| > 0 $,
 \beq\label{eq:ibp-local-1}
  \Big| \int_{I_j} e^{iaD(s)}\, ds \Big|
  \le
  \frac{1}{|a|}
    \Big\{
      \frac{2}{\delta} +  \frac{1}{\delta^2} \int_{I_j}|{d_0}'|\,ds
    \Big\}.
 \eeq
Summing over $j$ in \eqref{eq:ibp-local-1} using \eqref{eq:sum-1} yields
\eqref{eq:intbyparts1} and shows that the implicit constant
in \eqref{eq:intbyparts1} depends on $N$, $\delta$, and $\|{d_0}'\|_{L^1(I_j)}, j=1,\ldots,N$.
\epf

For the last bound  we need
\begin{lem}\label{lem:intbyparts2}
 Assume that $d_0$ is piecewise $W^{1,1} $  and bounded away from zero on $[0,L]$. Then, for all $a\in \R$, $a\not= 0$ and $R> 0$, we have
  \beq\label{eq:intbyparts2}
   \Big|
     \int\limits_{{|D|\lesssim R}} e^{-i\frac{a}{4D(s)}} \frac{1}{\abs{D(s)}}\, ds
   \Big|
   \lesssim
   \frac{R}{\abs{a}}
  \eeq
where $D(s):= \int_0^s d_0(\alpha)\, d\alpha$.
\end{lem}
\begin{rem}\label{rem:oscillatory-integral}
 First let's make precise what we mean by the integral in \eqref{eq:intbyparts2}:
 Choose a smooth cut--off function $\varphi:\R\to [0,1]$ with $\varphi(s)= 1$ for $\abs{s}\le 1$,
 $\varphi(s)= 0$ for $\abs{s}\ge 2$ and for $R>0$ put $\varphi_R(s):= \varphi(s/R)$. Then we set
 \beq
  \int\limits_{|D|\lesssim R} e^{-i\frac{a}{4D(s)}} \frac{1}{\abs{D(s)}}\, ds
  :=
  \int_0^L e^{-i\frac{a}{4D(s)}} \frac{\varphi_R(D(s))}{\abs{D(s)}}\, ds .
 \eeq
\end{rem}
\bpf[Proof of Lemma \ref{lem:intbyparts2}]
 Again let  $N<\infty$ be the number of intervals on which $d_0$ is absolutely continuous and
 $0=c_0<c_1<\ldots<c_N=L$ the partition of $[0,L]$ such that $d_0$ is absolutely continuous on each
 interval $I_k=[c_{k-1},c_k)$ and
 \beq\label{eq:sum-2}
  \int\limits_{|D|\lesssim R} e^{-i\frac{a}{4D(s)}} \frac{1}{\abs{D(s)}}\, ds
  = \sum_{j=1}^N  \int_{I_j} e^{-i\frac{a}{4D(s)}} \frac{\varphi_R(D(s))}{\abs{D(s)}}\, ds .
 \eeq
 Since
 \bdm
  e^{-i\frac{a}{4D(s)}}= \big(e^{-i\frac{a}{4D(s)}}\big)'  \frac{4 D(s)^2}{ia D'(s)}
 \edm
 and $D'=d_0$, an integration by parts gives
  \begin{multline}
   \int_{I_j} e^{-i\frac{a}{4D}} \frac{\varphi_R}{\abs{D}}\, ds
    =
     \frac{4}{ia}
     \int_{I_j} \big(e^{-i\frac{a}{4D}}\big)' \varphi_R(D) \frac{\abs{D}}{d_0}\, ds = \\
    =
     \frac{4}{ia}
     \left\{
       \left[
         e^{\frac{ia}{4D}} \varphi_R(D) \frac{\abs{D}}{d_0}
       \right]_{c_{j-1}}^{c_j}
       -
       \int_{I_j} e^{\frac{ia}{4D}}
       \Big[
         \varphi'_R (D)\frac{\abs{D}}{R} +\varphi_R(D)
         \big(
           \sgn(D) - \frac{\abs{D}{d_0}'}{{(d_0)}^2}
         \big)
       \Big]\, ds
     \right\}
  \end{multline}
 where $\sgn(r)= r/\abs{r}$ for $r\not=0$ and $\sgn(0)=0$. Because of
 $\max(|\varphi'_R|,\varphi_R)\lesssim \varphi_{2R}$ and $\varphi_{2R}(D)|D|\lesssim \varphi_{2R}(D)R$, taking the modulus above and using
 $\delta := \inf_{0\le s\le L} |d_0(s)|>0 $ yields
  \beq
   \left| \int_{I_j} e^{\frac{ia}{4D(s)}} \frac{\varphi_R(D(s))}{\abs{D(s)}}\, ds \right|
    \lesssim
     \frac{1}{|a|} \left\{
       \frac{R}{\delta}
       +
       \int_{I_j}\varphi_{2R}(D(s))\, ds + \frac{R}{\delta^2}\int_{I_j}  |d_0(s)|\, ds
     \right\} .
  \eeq
 Summing over $j$ we see
 \beq
  \begin{split}
  \left| \int_0^L e^{\frac{ia}{4D(s)}} \frac{\varphi_R(D(s))}{\abs{D(s)}}\, ds \right|
   &\lesssim
    \frac{1}{|a|}\big(R + \int_0^L \varphi_{2R}(D(s))\, ds\big) \\
   &\lesssim
    \frac{1}{|a|}\big(R + |\{ 0\le s\le L: |D(s)|\lesssim R \} | \big)
    \lesssim
    \frac{R}{|a|}
   \end{split}
 \eeq
 where in the last line we used  Lemma \ref{lem:sublevel}. This proves \eqref{eq:intbyparts2}.
\epf

The last bound we need is an immediate consequence from Lemma \ref{lem:intbyparts2}.
\begin{cor}\label{cor:small-D}
 Assume that $d_0$ is piecewise $W^{1,1} $  and bounded away from zero on $[0,L]$. Then, for all $a\in \R$, $a\not= 0$, we have
  \beq\label{eq:small-D}
   \Big|\,
     \int\limits_{{|D|\lesssim |a|}} e^{\frac{ia}{4D(s)}} \frac{1}{\abs{D(s)}}\, ds
   \Big|
   \lesssim
   \frac{1}{\max(\abs{a},1)}
  \eeq
where $D(s):= \int_0^s d_0(\alpha)\, d\alpha$.
\end{cor}
\bpf
Choose $R=\min(\abs{a},\|D\|_\infty)$, where
$\|D\|_\infty = \sup_{s\in[0,L]}|D(s)|\le \int_0^L |d_0(\nu)|\, d\nu<\infty$. Then, since
$|D(s)|\le \|D\|_\infty$ for all $s\in[0,L]$, the bound \eqref{eq:intbyparts2} from Lemma \ref{lem:intbyparts2} gives
 \bdm
  \Big|
   \int\limits_{{|D|\lesssim |a|}} e^{\frac{ia}{4D(s)}} \frac{1}{\abs{D(s)}}\, ds
  \Big|
  =
  \Big|\,
     \int\limits_{{|D|\lesssim R}} e^{\frac{ia}{4D(s)}} \frac{1}{\abs{D(s)}}\, ds
  \Big|
  \lesssim
  \frac{\min(|a|,\|D\|_\infty)}{|a|}
 \edm
which is \eqref{eq:small-D}.
\epf

\section{Reduction to Multilinear Estimates}\label{sec:multilinear}

The representations for $\calQ$ in Proposition~\ref{prop:rep} motivate the following
\begin{definition}\label{def:mult-sub-linear}
 For a piecewise $W^{1,1}$ dispersion profile $d_0$ on $[0,L]$ which is bounded away from zero we define the multi-sub-linear functions
   \begin{align}
    K^1(h_1,h_2,h_3,h_4)
     &:=
      \int_{\R^4}
       \frac{1}{L}\Big| \int_0^L e^{-i\frac{a(\eta)}{4D(s)}} \frac{ds}{|D(s)|} \Big|
       \prod_{\nu=1}^4 |h_\nu(\eta_\nu)|\, \delta(\eta_1-\eta_2+\eta_3-\eta_4) d\eta
       \label{eq:def-K1}\\
    \intertext{and }
    K^2(h_1,h_2,h_3,h_4)
     &:=
      \int_{\R^4}
       \frac{1}{L}\Big| \int_0^L e^{iD(s)a(\eta)}\, ds \Big|
       \prod_{\nu=1}^4 |h_\nu(\eta_\nu)| \, \delta(\eta_1-\eta_2+\eta_3-\eta_4) d\eta
       \label{eq:def-K2}
   \end{align}
 where $a(\eta):= \eta_1^2-\eta_2^2+\eta_3^2-\eta_4^2$ and $D(s)= \int_0^sd_0(\alpha)\,d\alpha$.
\end{definition}
\begin{rem}
   The multi-sub-linear functionals $K^1$ and $K^2$ are certainly defined for all (measurable) functions $h_1,h_2,h_3,h_4$, taking the value $\infty$, if necessary, but it seems unclear, at first, when they are finite and hence useful at all. For the convenience of the reader, we note that Proposition \ref{prop:mult-lin-bounds} below gives us the bounds
   \bdm
     K^1(h_1,h_2,h_3,h_4) \lesssim \prod_{n=1}^4 \|h_n\|
   \edm
   and
   \bdm
     K^2(h_1,h_2,h_3,h_4) \lesssim \prod_{n=1}^4 \|h_n\| ,
   \edm
   where the implicit constant depends only on the dispersion profile $d_0$ and is finite if $d_0$ is piecewise $W^{1,1}$ and bounded away from zero.
   Thus $K^1$ and $K^2$ are well--defined for all $h_n\in L^2(\R)$, $n=1,2,3,4$.
\end{rem}
A judicious use of the triangle inequality in \eqref{eq:rep-real-space} and \eqref{eq:rep-fourier-space} yields the following immediate
\begin{cor} \label{cor:Q-K-bound} The bounds
 \beq\label{eq:Q-K-bound}
 \begin{split}
  |Q(f_1,f_2,f_3,f_4)| & \le \frac{1}{4\pi } K^1(f_1,f_2,f_3,f_4) 
  \\
  |Q(f_1,f_2,f_3,f_4)| & \le \frac{1}{2\pi } K^2(\hatt{f}_1,\hatt{f}_2,\hatt{f}_3,\hatt{f}_4) 
 \end{split}
 \eeq
hold, where $\hatt{f}$ is the Fourier transform of $f$.
\end{cor}

The following multi-sub-linear bounds are crucial for our
result.
\begin{prop}[Multi--sub-linear bounds]\label{prop:mult-lin-bounds}
Assume that the dispersion profile $d_0$ is piecewise absolutely continuous on $[0,L]$ and bounded away from zero. Then
 \beq\label{eq:mult-lin}
  K^j(h_1,h_2,h_3,h_4)\lesssim \prod_{n=1}^4 \|h_n\|
 \eeq
 for $j=1,2$. If for some $l,m\in\{1,2,3,4\}$ one has
 $\tau:= \dist(\supp(h_l),\supp(h_m))>0$, then
 \beq\label{eq:mult-lin-loc}
  K^j(h_1,h_2,h_3,h_4)\lesssim \tau^{-1/2}\prod_{n=1}^4 \|h_n\|
 \eeq
 for $j=1,2$. The implicit constant in the above bounds depends only on $d_0$.
\end{prop}

As it is highly technical, we defer 
the proof of this Proposition to in 
Appendix~\ref{sec:multilinear-proof}
for the convenience of the reader.

\begin{rems}
 \itemthm \label{rem:boundedness}
  Of course, either bound above together with Proposition \ref{prop:mult-lin-bounds} in the appendix shows
 \beq
  |\calQ(f_1,f_2,f_3,f_4)| \lesssim \prod_{n=1}^4 \|f_n\| .
 \eeq
 Thus $\calQ:\calS^4\to\R$, a priori defined for $f_n$ in the Schwartz class $\calS$, extends by multi--linearity and continuity to a well-defined bounded four-linear functional
 $\calQ: (L^2(\R))^4\to \R$ and, by duality, since
 $\la g, Q(f_1,f_2,f_3) \ra =\calQ(g,f_1,f_2,f_3)$, the three--linear map $Q$
 also extends by multi-linearity and continuity to a three--linear map $Q:L^2(\R)^3\to L^2(\R)$ with
 \beq
  \|Q(f_1,f_2,f_3)\| \lesssim \prod_{j=1}^3 \|f_j\|.
 \eeq 
 This shows that $Q$, respectively $\calQ$, in \eqref{eq:Q}, respectively \eqref{eq:calQ}, are well--defined on $(L^2(\R))^4$, respectively, $(L^2(\R))^3$, as long as $d_0$ is piecewise $W^{1,1}$ and bounded away from zero.  \\[0.2em]
\itemthm  \label{rem:difficulty}
In fact, $Q$ and $\calQ$ are well-defined for a much larger class of dispersion profiles $d_0$, see \cite{HLexistence} for details, 
but for us the conditions of Proposition \ref{prop:mult-lin-bounds} are enough. In previous works, it had turned out to be 
advantageous to do another change of variables: Setting $r=D(s)$, one rewrites the definition of $Q$ in \eqref{eq:Q} and 
of $\calQ$ in \eqref{eq:calQ} as
     \bdm
       Q(f_1,f_2,f_3)= \int T^{-1}_{r}\big[ (T_{r}f_1)(\ol{T_{r}f_2})(T_{r}f_3)  \big]\, \mu(dr)
     \edm
     and
     \bdm
       \calQ(f_1,f_2,f_3,f_4)
       = \int\int  \ol{T_{r}f_1(x)}T_{r}f_2(x)\ol{T_{r}f_3(x)}T_{r}f_4(x)\, \mu(dr) dx
     \edm
    where $\mu$ is a probability measure given by
     \bdm
      \mu(A) = \frac{1}{L}\int_0^L \id_A(D(s))\, ds ,
     \edm
    and still $D(s)= \int_0^s d_0(\alpha)\, d\alpha$. When one can directly work with $\calQ$, this leads to some simplifications 
since the one-dimensional Strichartz inequality \cite{Foschi,HZ06,Strichartz} yields boundedness of $\calQ$ as soon as $\mu$ has 
a density $\psi\in L^2$, see \cite{HLexistence}. Under rather general assumptions on $d_0$, see the discussion in \cite{HLexistence}, 
the measure $\mu$ indeed has a density $\psi$, i.e., $\mu(dr)= \psi(r)\, dr$.  Further, $\psi\in L^2$ is implied by
    $d_0^{-1}\in L^1$([0,L]), see Lemma 1.4 and Lemma B.1 in \cite{HLexistence} for details. However, unlike in the existence proof 
in \cite{HLexistence}, we have to work with $K^1$ and $K^2$ and not directly with $\calQ$  mainly because of 
Theorem \ref{thm:exp-twisted-bound} below and we have to use an `integration by parts' type argument to control the oscillatory 
integrals in the definition of $K^1$ and $K^2$.
    Using the probability measure $\mu(dr)=\psi(r)dr$, the oscillatory integrals \eqref{eq:osc-integ-1}, 
respectively \eqref{eq:osc-integ-2}, in the definition of $K^1$, respectively $K^2$, can be rewritten as
     \beq
      \int e^{-i\frac{a}{4r}} \frac{1}{|r|} \psi(r)\, dr ,
      \text{ respectively }
      \int e^{i r a}  \psi(r)\, dr .
     \eeq
    Thus the necessary oscillatory bounds would follow from suitable smoothness properties of $\psi$. In the model case \eqref{eq:d0-simplest}, which is the case most often studied so far, this density turns out to be given by $\psi=\id_{[0,1]}$, which results in \emph{considerable simplifications} in the necessary oscillatory and multi-linear bounds for $K^1$ and $K^2$ in \cite{EHL}. However, already this model case shows that $\psi$ \emph{must have jump discontinuities} in general.
   Unlike in the existence proof of \cite{HLexistence}, where only $L^p$ properties of $\psi$ were needed, which, in turn, follow from $L^{p-1}$ properties of $d_0^{-1}$, see Lemma 1.4 in \cite{HLexistence}, we found it rather inconvenient to deduce the necessary differentiability properties of $\psi$ from natural conditions on $d_0$. Instead, it turned out to be easier to avoid the above change of variables and work with the oscillatory integrals in \eqref{eq:osc-integ-1} and \eqref{eq:osc-integ-2} and the Definition \ref{def:mult-sub-linear} directly, 
    based on, in particular, Lemma \ref{lem:intbyparts2} and Corollary \ref{cor:small-D}.
\end{rems}

Lastly, as in \cite{EHL}, in the proof of the exponential decay estimates a key ingredient are multi-linear bounds for suitably exponentially twisted versions of $K^1$ and $K^2$ which are \emph{uniform} in the twist. First the definition of the twisted functionals:
\begin{definition}\label{def:Ktwisted}
 Let $F:\R\to[0,\infty)$ be  a symmetric function which obeys the triangle inequality, i.e., $F(-x)= F(x)$ and
 $F(x+y)\le F(x)+F(y)$ for all $x,y\in\R$. The exponentially twisted functionals $K^1_F$ and $K^2_F$ are defined by \beq
  K^j_F(h_1,h_2,h_3,h_4) := K^j(e^{F}h_1,e^{-F}h_2,e^{-F}h_3,e^{-F}h_4)\, \quad j=1,2.
 \eeq
\end{definition}
Since $K^1$ and $K^2$ are highly nonlocal, it seems at first rather surprising that $K^j_F$ is even finite for an unbounded exponential twist $F$, since the factor $e^F$ in the first entry of the definition of $K^j_F$ is potentially \emph{exponentially large}. But exponentially twisted sub--linear functionals $K^j_F$ are not only well-defined but, even more surprisingly, they are bounded by $F=0$, as the following theorem shows, whose proof is astonishingly simple. Together with Proposition \ref{prop:mult-lin-bounds} it yields multi--linear bounds for $K^j_F$ which are \emph{uniform in the exponential twist} $F$, as long as $F$ is even and obeys the triangle inequality.

\begin{thm}\label{thm:exp-twisted-bound}
 Let $F:\R\to[0,\infty)$ be  a symmetric function which obeys the triangle inequality, i.e.,
 $F(x+y)\le F(x)+F(y)$ for all $x,y\in\R$. Then for all $h_l\in L^2(\R)$, $l=1,\ldots,4$,
 \beq
   K^j_F(h_1,h_2,h_3,h_4) \le K^j(h_1,h_2,h_3,h_4),\quad j=1,2.
 \eeq
\end{thm}
\bpf
 By the definition of $K^1$ one has
 \beq\label{eq:crucial}
  \begin{split}
    K^1_F(&h_1,h_2,h_3,h_4) \\
     &= \int_{\R^4}
       \frac{1}{L}\Big| \int_0^L e^{-i\frac{a(\eta)}{4D(s)}} \frac{ds}{|D(s)|} \Big|
       e^{F(\eta_1)-F(\eta_2)-F(\eta_3)-F(\eta_4)}
       \prod_{j=1}^4 |h_j(\eta_j)|\, \delta(\eta_1-\eta_2+\eta_3-\eta_4) d\eta
  \end{split}
 \eeq
 On the support of the measure $\delta(\eta_1-\eta_2+\eta_3-\eta_4) d\eta $ on $\R^4$ one has
 $ \eta_1-\eta_2+\eta_3-\eta_4 = 0 $ and hence
 \bdm
  F(\eta_1)= F(\eta_2-\eta_3 + \eta_4) \le F(\eta_2) + F(\eta_3) + F(\eta_4)
 \edm
 since $F$ is symmetric and obeys the triangle inequality. In particular, one has
 \bdm
  e^{F(\eta_1)-F(\eta_2)-F(\eta_3)-F(\eta_4)}\le 1
 \edm
 in the integral in \eqref{eq:crucial} and the bound $K^1_F(h_1,h_2,h_3,h_4) \le K^1(h_1,h_2,h_3,h_4)$ follows immediately. The same proof shows $K^2_F(h_1,h_2,h_3,h_4) \le K^2(h_1,h_2,h_3,h_4)$.
\epf

\section{The Exponential Decay}\label{sec:expdecay}
\subsection{The proof of exponential decay: real space}
In this section, we prove an $L^2$--version of the exponential decay, namely
\begin{prop} \label{prop:x-space-L2-exp-decay} Assume that $d_0$ is piecewise $W^{1,1}$ and bounded away from zero on $[0,L]$ and that $f\in H^1$ is a solution of \eqref{eq:DMS-weak} for some $\omega>0$ and $d_{\mathrm{av}}> 0$ {\rm (}respectively, $f\in L^2$ is a solution of \eqref{eq:DMS-weak} for $\d_\mathrm{av}=0${\rm )}. Then there exists $\mu>0$ such that the exponentially weighted function $x\mapsto e^{\mu|x|}f(x)$ is still in $L^2(\R)$.
\end{prop}
To show that dispersion management solitons have exponential decay we start with the weak form of the dispersion management equation \eqref{eq:DMS-weak}. Let $\xi\in \calC_b^\infty(\R)$ be a real-valued  bounded infinitely often differentiable function whose derivative is bounded. Choosing $g=\xi^2 f$ in \eqref{eq:DMS-weak} one sees
 \beq
  -\omega \|\xi f\|^2 = -\omega \la \xi^2 f,f \ra
   = d_{\text{av}}\la (\xi^2 f)',f' \ra - \calQ(\xi^2 f, f,f,f)
 \eeq
Since $-\omega \|\xi f\|^2$ is real we can take the real part above to get
 \beq\label{eq:bound1}
  -\omega \|\xi f\|^2
    = d_{\text{av}}\re\la (\xi^2 f)',f' \ra -  \re\calQ(\xi^2f,f,f,f)
 \eeq
We deal with the real part of the kinetic energy term with the help of the following simple

\begin{lem}
 Let $\xi$ be a real-valued bounded smooth function whose derivative is bounded and $f\in H^1(\R)$, then $\xi f,\xi^2 f\in H^1(\R)$ and
   \beq\label{eq:IMS}
     \re\la (\xi^2 f)',f' \ra = \la (\xi f)',(\xi f)' \ra - \la f, |\xi'|^2 f \ra
   \eeq
\end{lem}
\begin{proof}
 This is, of course, nothing but a one-dimensional version of the well known IMS localization formula for the kinetic energy, see \cite{CFKS}, and it holds in much greater generality. We give a short proof for the convenience of the reader.

Due to the smoothness and boundedness assumptions on $\xi$ it is clear that $\xi f$ and $\xi^2f$ are in $H^1(\R)$. By density, it is also enough to prove \eqref{eq:IMS} for $f\in\calC^\infty_0(\R)$. Let $H$ be an operator, in our case it will be given by $H=-\partial_x^2$, and assume that $\calC^\infty_0(\R)$ is in the domain of $H$. Then
 \beq
   [H,\xi] = H\xi - \xi H
 \eeq
and the iterated commutator is given by
 \beq\label{eq:double-commutator1}
   [[H,\xi],\xi] = H\xi^2 - 2\xi H\xi + \xi^2 H.
 \eeq
Hence
 \beq\label{eq:double-commutator2}
  \begin{split}
   \re \la \xi^2 f, H f\ra &= \frac{1}{2} \la f, (\xi^2H +H\xi^2) f \ra
     = \la f, \xi H\xi f\ra - \frac{1}{2} \la f,  [[H,\xi],\xi]  f \ra \\
     & = \la \xi f,  H\xi f\ra - \frac{1}{2} \la f,  [[H,\xi],\xi]  f \ra
  \end{split}
 \eeq
If $H=-\partial_x^2$, then one calculates $[[H,\xi],\xi]= 2|\xi'|^2$ and thus
\eqref{eq:double-commutator2} yields
 \beq
  \begin{split}
   \re \la (\xi^2 f)', f'\ra &= \re \la \xi^2 f, -\partial_x^2 f\ra
     = \la \xi f, -\partial_x^2(\xi f)\ra -  \la f,  |\xi'|^2 f \ra \\
     & = \la (\xi f)', (\xi f)'\ra -  \la f,  |\xi'|^2 f \ra ,
  \end{split}
 \eeq
at least for $f\in\calC^\infty_0(\R)$ and then by density for all $f\in H^1(\R)$.
\end{proof}

 Using \eqref{eq:IMS} in \eqref{eq:bound1}, dropping the non-negative term
$d_{\text{av}}\la (\xi f)',(\xi f)'\ra $, bounding $\re\calQ$ by $|\calQ|$, and using Corollary \ref{cor:Q-K-bound}, one gets
 \beq\label{eq:bound2}
    \|\xi f\|^2 \le
    \omega^{-1}(d_{\text{av}} \la f, |\xi'|^2f \ra + |\calQ(\xi^2f,f,f,f)| )
    \lesssim \la f, |\xi'|^2f \ra + K^1(\xi^2f,f,f,f)
 \eeq
where the implicit constant depends only on $\omega, d_{\text{av}}$, and $L$.

No we come to the choice of $\xi$:  Choose a smooth  function $\chi$ with $0\le \chi\le 1$, $\chi(x)=0$ for $|x|\le 3/2$\footnote{That $\chi=0$ on $(-3/2,3/2)$ is mainly for convenience in the proof of Lemma \ref{lem:localization-error2}.}, $\chi(x)=1$ for $|x|\ge 2$, and for $\tau>0$ put $\chi_\tau(x):= \chi(x/\tau)$.
Furthermore, for $\veps,\mu>0$ set
  \beq\label{eq:F}
    F(x)=F_{\mu,\veps}(x)= \frac{\mu|x|}{1+\veps|x|}.
  \eeq
Then we put $\xi:= e^{F_{\mu,\veps}}\chi_\tau $, i.e., $\xi$ is an exponentially
weighted cut--off function which is smooth, since $F_{\mu,\veps}$ is smooth on the support of $\chi_\tau$.

\begin{lem}\label{lem:localization-error1} With $\xi= e^{F_{\mu,\veps}}\chi_\tau$ and the choice $\mu=\tau^{-1}$ one has
  \bdm
    \la f, |\xi'|^2 f\ra \lesssim \frac{1}{\tau^2}\big( \|f\|^2 + \| e^{F_{\mu,\veps}} \chi_\tau f\|^2 \big)
  \edm
 uniformly in $\tau,\veps>0$, where the implicit constant depends only on $\|\chi'\|_\infty$.
\end{lem}

\bpf
Recall that $F=F_{\mu,\veps}$ is given by \eqref{eq:F}, $\xi= e^{F}\chi_\tau$ where $\chi$ is a smooth cut--off function with $0\le \chi\le 1$, $\chi(x)=0$ for $|x|\le 3/2$, $\chi(x)=1$ for $|x|\ge 2$, and
$\chi_\tau(x)= \chi(x/\tau)$, and $\xi=e^F\chi_\tau$.
 In view of $|F'(x)|= |F_{\mu,\veps}'(x)|= \mu(1+\veps|x|)^{-2}\le \mu$ for all $\veps>0$ and $x\not=0$ and $(\chi_\tau)'=\tau^{-1}(\chi')_\tau$,
 \bdm
  \xi' =(e^F\chi_\tau)' \le \mu e^F \chi_\tau + \frac{1}{\tau} e^F (\chi')_\tau .
 \edm
Thus
 \bdm
  |\xi'|^2 \le 2 e^{2F}(\mu^2 \chi_\tau^2 + \tau^{-2}(\chi')_\tau^2) .
 \edm
Moreover, because $\chi'(x)=0$ if $|x|\ge 2$, we have $(\chi')_\tau(x)=0$ for $|x|\ge 2\tau$ and thus
 \bdm
  e^{2F}(\chi')_\tau^2 \le e^{2F(2\tau)}\|\chi'\|_\infty^2 \le e^{4\mu\tau} \|\chi'\|_\infty^2 .
 \edm
In particular,
 \bdm
 \begin{split}
  \la f, |\xi'|^2 f\ra
   &
    \le
     2 \big( \mu^2 \|e^{F}\chi_\tau f\|^2 + e^{4\mu\tau}\tau^{-2} \|\chi'\|_\infty^2\|f\|^2 \big) 
    \le
     \frac{2\max(1, e^4\|\chi'\|_\infty^2)}{\tau^2}
     \big( \|e^{F}\chi_\tau f\|^2 + \|f\|^2 \big)
 \end{split}
 \edm
 for $\mu=\tau^{-1}$, which proves Lemma \ref{lem:localization-error1}.
\epf

\begin{lem}\label{lem:localization-error2} With the choice $\mu=\tau^{-1}$ one has
  \bdm
    K^1(e^{2F_{\mu,\veps}}\chi_\tau^2f,f,f,f)
     \lesssim
       \| e^{F_{\mu,\veps}}\chi_\tau f\|^4
       + \|e^{F_{\mu,\veps}}\chi_\tau f\|^3
       + c_f(\tau) \|e^{F_{\mu,\veps}}\chi_\tau f\|^2
       + c_f(\tau) \|e^{F_{\mu,\veps}}\chi_\tau f\|
  \edm
where the implicit constant depends only on $\|f\|$ and $c_f(\tau)$ denotes terms which, for fixed $f\in L^2$, go to zero as $\tau\to\infty$ uniformly in $\veps>0$.
\end{lem}

\bpf
We will continue to abbreviate $F_{\mu,\veps}$, given by \eqref{eq:F}, by $F$. Note that $\chi_\tau\le 1$, so by monotonicity and Definition \ref{def:mult-sub-linear} of $K^1$,
 \bdm
   K^1(e^{2F}\chi_\tau^2f,f,f,f)\le K^1(e^{2F}\chi_\tau f,f,f,f) .
 \edm
Define $h:= e^F f$ and $h_> := e^{F}\chi_\tau f = \chi_\tau h$ so that $h_>$ is supported on $\{x\in\R\vert\, |x|\ge 3\tau/2\}$. Then, using the twisted functional $K^1_F$ given in \eqref{def:Ktwisted}, one sees
 \bdm
 \begin{split}
  K^1(e^{2F}\chi_\tau f,f,f,f)
   &
    =
    K^1(e^{F}h_>,e^{-F}h,e^{-F}h,e^{-F}h) 
    =
    K^1_F(h_>,h,h,h)
   \le
    K^1 (h_>,h,h,h)
 \end{split}
 \edm
where we also used Theorem \ref{thm:exp-twisted-bound}, which applies since $F$ given by \eqref{eq:F} obeys the triangle inequality.
Thus, in order to prove Lemma \ref{lem:localization-error2}, we have to show that for $\mu=\tau^{-1}$ the bound
 \beq\label{eq:K1-main-bound}
    K^1(h_>,h,h,h)
     \le
       C (\|h_>\|^4 + \|h_>\|^3)
       + c_f(\tau) (\|h_>\|^2 + \|h_>\| )
  \eeq
holds for some constant $C$ which depends only on $\|f\|$ and where $c_f(\tau)$ denotes a term which, for fixed $f\in L^2$, goes to zero as $\tau\to\infty$ uniformly in $\veps>0$.

Continuing with the proof of \eqref{eq:K1-main-bound} recall that $h_<= h- h_> = (1-\chi_\tau)h$. Using sub-linearity of $K^1$ one sees
 \beq\label{eq:sublinear}
  \begin{split}
  K^1( h_>, h,h,h)
  &
   \le
    K^1(h_>,h_>,h_>,h_>) + K^1(h_>,h_>,h_>,h_<) \\
  & \phantom{\le ~}
    + K^1(h_>,h_>,h_<,h_<)
    +  K^1(h_>,h_<,h_<,h_<) \\
  &
    \phantom{\le~}
    + \text{cyclic permutations of the last } 3 \text{ entries of } K^1.
  \end{split}
 \eeq
For the first term in \eqref{eq:sublinear} Proposition \ref{prop:mult-lin-bounds} yields
 \beq\label{eq:0smallh}
   K^1(h_>,h_>,h_>,h_>)  \le C \|h_>\|^4  \eeq
for some constant $C$. Also
 \bdm
   K^1(h_>,h_>,h_>,h_<)  \le C \|h_>\|^3\|h_<\|.
 \edm
Put $f_<= (1-\chi_\tau)f$ and note that since $f_<$ supported on the set $[-2\tau,2\tau]$ and $F$ is increasing, we have
 \beq\label{eq:local-h-bound}
   \|h_<\| = \|e^Ff_<\|\le e^{F(2\tau)}\|f_<\| \le e^{2\mu\tau}\|f\|
 \eeq
since $F(2\tau)= F_{\mu,\veps}(2\tau)\le 2\mu\tau$ uniformly in $\veps>0$. Thus
 \beq\label{eq:1smallh}
   K^1(h_>,h_>,h_>,h_<)  \le C e^{2\mu\tau} \|f\| \|h_>\|^3 = Ce^2 \|f\| \|h_>\|^3
 \eeq
when $\mu=\tau^{-1}$ and the same bound holds for cyclic permutations of the last 3 entries in $K^1$.

To bound the third and fourth terms in \eqref{eq:sublinear} let $\id_A$ be the characteristic function of a set $A\subset\R$ and define
 \begin{align*}
   h_\ll &:= \id_{[-\tau,\tau]} h \\
   h_\sim & := h- h_>-h_<= (1-\chi_\tau)h- h_< = (1-\chi_\tau)\id_{[-\tau,\tau]^c} h
 \end{align*}
and similarly for $f_\ll$ and $f_\sim$. Since $f_\ll$ is supported on $[-\tau,\tau]$ and $f_\sim$ is supported on the `annulus' $[-2\tau,2\tau]\setminus [-\tau,\tau]$, we have in addition to \eqref{eq:local-h-bound} the bounds
 \beq\label{eq:more-local-h-bounds}
  \|h_\ll\|\le e^{\mu\tau}\|f\| \text{ and } \|h_\sim\|\le e^{2\mu\tau} \|f_\sim\|.
 \eeq
 Note
 \begin{align*}
  K^1(h_>,h_>,h_<,h_<)
  &
   = K^1(h_>,h_>,h_\sim+h_\ll,h_<)
   \le
     K^1(h_>,h_>,h_\sim,h_<) + K^1(h_>,h_>,h_\ll,h_<).
 \end{align*}
By \eqref{eq:mult-lin} together with \eqref{eq:local-h-bound} and \eqref{eq:more-local-h-bounds} we have
 \bdm
  K^1(h_>,h_>,h_\sim,h_<)\le C \|h_>\|^2 \|h_\sim\|\|h_<\| \le C e^{4\mu\tau} \|f_\sim\| \|f\| \|h_>\|^2 .
 \edm
Since the supports of $h_>$ and $h_\ll$ are separated by at least $\tau/2$, the bound \eqref{eq:mult-lin-loc} yields
 \bdm
 K^1(h_>,h_>,h_\ll,h_<) \le \frac{C}{\sqrt{\tau}} \|h_>\|^2 \|h_\ll\|\|h_<\|
 \le
  \frac{Ce^{3\mu\tau}}{\sqrt{\tau}} \|f\|^2 \|h_>\|^2
 \edm
where we also used \eqref{eq:local-h-bound} and \eqref{eq:more-local-h-bounds}. Together this gives
 \bdm
  \begin{split}
  K^1(h_>,h_>,h_<,h_<)
  &
   \le C\|f\|\big(e^{4\mu\tau}\|f_\sim\| + \frac{e^{3\mu\tau}}{\sqrt{\tau}} \|f\|\big)\|h_>\|^2
   \lesssim \big(\|f_\sim\| + \tau^{-1/2} \big)\|h_>\|^2
  \end{split}
 \edm
for $\mu=\tau^{-1}$. Since $f\in L^2$, $\|f_\sim\|$ goes to zero as $\tau\to\infty$. Thus, with $c_f(\tau):= \|f_\sim\| + \tau^{-1/2} $, the above
shows
 \beq\label{eq:2smallh}
  K^1(h_>,h_>,h_<,h_<)\le c_f(\tau) \|h_>\|^2
 \eeq
where, for fixed $f\in L^2$, the term $c_f(\tau)$ goes to zero as $\tau\to\infty$ \emph{uniformly } in $\veps>0$ and the same bound holds for cyclic permutations of the last 3 entries in $K^1$. 

For $K^1(h_>,h_<,h_<,h_<)$ note that an argument virtually identical to the one for bounding $K^1(h_>,h_>,h_<,h_<)$ yields
 \bdm\label{eq:3smallh}
  \begin{split}
   K^1(h_>,h_<,h_<,h_<)
   &
    \le K^1(h_>,h_\sim,h_<,h_<)+K^1(h_>,h_\ll,h_<,h_<) \\
   &
    \le
     C\|h_>\|\|h_\sim\| \|h_<\|^2 + \frac{C}{\sqrt{\tau}} \|h_>\|\|h_\ll\|\|h_<\|^2   \\
   &
    \le
    C\|f\|^2 \big( \|f_\sim\| e^{6\mu\tau}+ \tau^{-1/2}e^{5\mu\tau}  \|f\| \big) \|h_>\|
    \lesssim
    c_f(\tau) \|h_>\|
  \end{split}
 \edm
for $\mu=\tau^{-1}$, where again, for fixed $f\in L^2$, the term $c_f(\tau)$ goes to zero as $\tau\to\infty$ \emph{uniformly } in $\veps>0$.

Plugging \eqref{eq:0smallh}, \eqref{eq:1smallh}, \eqref{eq:2smallh}, and \eqref{eq:3smallh} into \eqref{eq:sublinear} shows \eqref{eq:K1-main-bound} and hence the Lemma.
\epf

\begin{rem}
 While at first it seems surprising that one can bound $K^1(h_>,h,h,h)$ solely in terms of $\|h_>\|$, using sublinearity one can easily see that $K^1(h_>,h,h,h)$  can be bounded in terms of $\|h_>\|$ \emph{together with} $\|h_<\|$ where 
 $h_<= h- h_> = (1-\chi_\tau)h$. Since $h_<$ has support in $[-2\tau,2\tau]$, one has $\|h_<\|\le e^{2\mu\tau}\|f_<\|$ and this is bounded by $\|f\|$ if one chooses $\mu=\tau^{-1}$. This is why we have to make this choice for $\mu$. More refined arguments are needed to see that certain terms in the bound go to zero as $\tau\to\infty$, as seen in the arguments above.
\end{rem}

These lemmata, together with \eqref{eq:bound2}, yield the $L^2$--exponential decay of $f$: 
\begin{proof}[Proof of Proposition \ref{prop:x-space-L2-exp-decay}]
Using Lemma \ref{lem:localization-error1} and \ref{lem:localization-error2} on the right hand side of \eqref{eq:bound2}, we get
 \bdm
   \|e^{F_{\mu,\veps}}\chi_\tau f\|^2
     \le C \big(\|e^{F_{\mu,\veps}}\chi_\tau f\|^4 + \|e^{F_{\mu,\veps}}\chi_\tau f\|^3 \big)
         + c_f(\tau) \big( \|e^{F_{\mu,\veps}}\chi_\tau f\|^2 + \|e^{F_{\mu,\veps}}\chi_\tau f\| \big)
         + c_f(\tau)
 \edm
where we always put $\mu=\tau^{-1}$ and $c_f(\tau)$ denotes terms which, for fixed $f\in L^2$, go to zero as $\tau\to\infty$ uniformly in $\veps>0$. Defining $\|f\|_{\tau,\veps}:= \|e^{F_{\tau^{-1},\veps}}\chi_\tau f\|$ we can rewrite the above bound as
 \beq\label{eq:bound3}
  (1-c_f(\tau))\|f\|_{\tau,\veps}^2
  -C \big(\|f\|_{\tau,\veps}^3 + \|f\|_{\tau,\veps}^4 \big)
  - c_f(\tau) \|f\|_{\tau,\veps}
  \le c_f(\tau)
 \eeq

Consider the function $\wti{G}:[0,\infty)\to \R$, given by $\wti{G}(\nu):= \nu^2/2 -C(\nu^3+\nu^4)$.  It is clear that there is a $\tilde{\nu}>0$ such that $\wti{G}(\nu)$ is positive for small $0<\nu<\tilde{\nu}$ and negative for $\nu>\tilde{\nu}$ and $\wti{G}$ has a (single) strictly positive maximum at a point $0<\tilde{\nu}_{m}<\tilde{\nu}$.
By continuity, there is a $\delta>0$ such that the function $G:[0,\infty)\to\R$ given by
 $G(\nu):= \wti{G}(\nu)- \delta \nu$
still has a strictly positive maximum, say $G_{\text{max}}>0$ and, moreover, there are $0<\nu_1<\nu_2$ such that $G(\nu)\le G_{\text{max}}/2$ implies $\nu\in [0,\nu_1]\cup[\nu_2,\infty)$.

Now choose $\tau>0$ so large that $c_f(\tau)\le \min(\delta,1/2)$. Then \eqref{eq:bound3} can be rewritten as
 \beq\label{eq:punchline}
  G(\|f\|_{\tau,\veps}) \le c_f(\tau).
 \eeq
The bound \eqref{eq:punchline} is our main tool, since it allows us to deduce the \emph{finiteness} of $\|f\|_{\tau,\veps}$ in the limit $\veps\to 0$ from the \emph{smallness} of $\|f\|_{\tau,\veps}$ in the limit\footnote{The reader who is uncomfortable with first letting $\veps\to \infty$ in order to deduce something useful for the limit $\veps\to 0$, should replace $\veps$ by $\alpha$ in the argument.} $\veps\to\infty$.
This is a two punch argument and goes as follows:

As a preparation: if we choose $\tau$ so large that $c_f(\tau)< G_{\text{max}}/2$, then \eqref{eq:punchline} shows that $\|f\|_{\tau,\veps}$ must be trapped in the region $[0,\nu_1]\cup[\nu_2,\infty)$. Note that since neither $G$ nor the $c_f(\tau)$ term on the right hand side of \eqref{eq:punchline} depend on $\veps>0$, this is uniformly in $\veps>0$.

The first punch of our argument is that the map $0<\veps\mapsto \|f\|_{\tau,\veps}$ is \emph{continuous}. By continuity, as long as $\tau$ is so large that
$c_f(\tau)< \min(\delta,1/2,G_{\text{max}}/2)$, the bound \eqref{eq:punchline} ensures that, as long as $\|f\|_{\tau,\veps_0}\le\nu_1$ \emph{for some} $\veps_0>0$, one has $\|f\|_{\tau<\veps_0}\le\nu_1$ \emph{for all} $\veps>0$, because in order to go from  $\|f\|_{\tau,\veps_0}\le\nu_1$ to $\|f\|_{\tau,\veps_1}\ge\nu_2$, the bound \eqref{eq:punchline} would have to violated at some intermediate value of $\veps>0$ by the intermediate value theorem for continuous functions.

The second, and final, punch of our argument is that by choosing $\tau$ be even larger, if necessary, we can assume that $\|\chi_\tau f\|<\nu_1$ and letting $\veps\to\infty$, one sees
 \beq
  \lim_{\veps\to\infty}\|f\|_{\tau,\veps}
    = \lim_{\veps\to\infty} \|e^{F_{\tau^{-1},\veps}}\chi_\tau f\|
    = \|\chi_\tau f\| <\nu_1
 \eeq
where the last equality follows from Lebesgue's dominated convergence convergence theorem, since $\lim_{\veps\to\infty} F_{\tau^{-1},\veps}=0$ pointwise and $F_{\tau^{-1},\veps}\le \tau^{-1}$ for $\veps\ge 1$. So for \emph{large enough} $\veps>0$ we indeed have  $\|f\|_{\tau,\veps}<\nu_1$ and thus, by the above argument, we have $\|f\|_{\tau,\veps}<\nu_1$ for all $\veps>0$. Hence
 \beq
  \|e^{F_{\mu,0}} \chi_\tau f\|=\lim_{\veps\to 0} \|f\|_{\tau,\veps} = \sup_{\veps>0}\|f\|_{\tau,\veps}\le\nu_1 ,
 \eeq
by the monotone convergence theorem. Since $x\mapsto e^{\tau^{-1}|x|}= e^{F_{\tau^{-1},0}(x)}$ is locally bounded and $1-\chi_\tau$ has support in $[-2\tau,2\tau]$, the above bound shows that
 \beq
  \|e^{F_{\tau^{-1},0}} f\|
  \le \|e^{F_{\tau^{-1},0}} (1-\chi_\tau) f\| + \|e^{F_{\tau^{-1},0}} \chi_\tau f\|
  \le e^{2}\| f\| + \nu_1 <\infty
 \eeq
which finishes the proof.
\end{proof}

\subsection{The proof of exponential decay: Fourier space}
In this section, we prove an $L^2$--version of the exponential decay for the Fourier transform, namely
\begin{prop} \label{prop:k-space-L2-exp-decay} Assume that $d_0$ is piecewise $W^{1,1}$ and bounded away from zero on $[0,L]$ and that $f\in H^1$ is a solution of \eqref{eq:DMS-weak} for some $\omega>0$ and $d_{\mathrm{av}}> 0$ {\rm (}respectively, $f\in L^2$ is a solution of \eqref{eq:DMS-weak} for $\d_\mathrm{av}=0${\rm)}. Then there exists $\mu>0$ such that
$x\mapsto e^{\mu|x|}\hat{f}(x)\in L^2(\R)$.
\end{prop}

\bpf
The proof of this Proposition is actually a bit simpler than the one for Proposition \ref{prop:x-space-L2-exp-decay}.

Let $p=-i\partial_x$ the the one-dimensional momentum operator and let the cut--off functions $\chi$, $\chi_\tau$, and $F=F_{\mu,\veps}$ from \eqref{eq:F} be as before in the proof of Proposition \ref{prop:x-space-L2-exp-decay}. Furthermore, let $\xi$ now the the \emph{operator} given by
$\xi:= e^{F(P)}\chi_\tau(P)$, which is, of course, given by multiplying by $e^{F(k)}\chi_\tau(k)$ in Fourier space and lastly, given $f\in L^2$ choose $g=\xi^2f$. If $f$ solves \eqref{eq:DMS-stationary}  then \eqref{eq:DMS-weak} yields
 \begin{align*}
  -\omega \|e^{F}\chi_\tau \hatt{f}\|^2
  &
   = -\omega\la \hatt{g},\hatt{f} \ra
   = -\omega\la g, f \ra
   =
    \d_\mathrm{av}\la g',f' \ra
    - \calQ(g,f,f,f)
 \end{align*}
Since $\xi$ is symmetric and commutes with derivatives, one has
 \bdm
  \la g',f' \ra = \la \xi^2 f', f' \ra = \la (\xi f)', (\xi f)' \ra \ge 0
 \edm
and, because $d_\mathrm{av}\ge 0$, we can drop the kinetic energy term to see
 \beq
  \|e^{F}\chi_\tau \hatt{f}\|^2 \lesssim Q(\xi^2f,f,f,f) \le |Q(\xi^2f,f,f,f)|
  \lesssim
   K^2(e^{2F}\chi_\tau^2 \hatt{f},\hatt{f},\hatt{f},\hatt{f}).
 \eeq
For $K^2$ we have a similar bound as for $K^1$ in Lemma  \ref{lem:localization-error2}:
\begin{lem}\label{lem:localization-error-Fourier} With the choice $\mu=\tau^{-1}$ one has
  \bdm
    K^2(e^{2F_{\mu,\veps}}\chi_\tau^2\hatt{f},\hatt{f},\hatt{f},\hatt{f})
     \lesssim
       \| e^{F_{\mu,\veps}}\chi_\tau \hatt{f}\|^4
       + \|e^{F_{\mu,\veps}}\chi_\tau \hatt{f}\|^3
       + c_f(\tau) \|e^{F_{\mu,\veps}}\chi_\tau \hatt{f}\|^2
       + c_f(\tau) \|e^{F_{\mu,\veps}}\chi_\tau \hatt{f}\|
  \edm
where the implicit constant depends only on $\|f\|$ and $c_f(\tau)$ denotes terms which go to zero as $\tau\to\infty$ uniformly in $\veps>0$.
\end{lem}
\bpf
With the bounds from Theorem \ref{thm:exp-twisted-bound} and Proposition \ref{prop:mult-lin-bounds} this is proven exactly the same way as Lemma \ref{lem:localization-error2}.
\epf
This Lemma and the above argument show
 \bdm
  \|e^{F_{\mu,\tau}}\chi_\tau \hatt{f}\|
  \le
    C\| e^{F_{\mu,\veps}}\chi_\tau \hatt{f}\|^3
    + C\|e^{F_{\mu,\veps}}\chi_\tau \hatt{f}\|^2
     + c_f(\tau)\|e^{F_{\mu,\veps}}\chi_\tau \hatt{f}\|
       + c_f(\tau)
 \edm
 for some constant $C$, or introducing $H(\nu):= \nu/2- C(\nu^3+\nu^2)$, we have
 \beq\label{eq:punchlineFourier}
  H(\|e^{F_{\mu,\veps}}\chi_\tau \hatt{f}\|)\le c_f(\tau)
 \eeq
 for $\mu=\tau^{-1}$ and all $\tau$ large enough such that $o(1)\le 1/2$. Because $\hatt{f}\in L^2$, we can deduce from \eqref{eq:punchlineFourier} with a similar argument as in the proof of Proposition \ref{prop:x-space-L2-exp-decay} that for some large enough $\tau>0$ the map $k\mapsto e^{\tau^{-1}|k|}\hatt{f}(k)\in L^2$. In fact, this is the original proof in \cite{EHL} for $d_{\mathrm{av}}=0$. This proves Proposition \ref{prop:k-space-L2-exp-decay}.
\epf

\subsection{Proof of exponential decay: Uniform bounds.}
\label{sec:uniform-exp-decay1}
Given Propositions \ref{prop:x-space-L2-exp-decay} and \ref{prop:k-space-L2-exp-decay}, the proof of Theorem \ref{thm:exp-decay} is a simple interpolation. Let $f\in L^2$ be a solution of \eqref{eq:DMS-weak}. Then Proposition \ref{prop:k-space-L2-exp-decay} implies $f\in H^\infty(\R)$, in particular, it is infinitely often differentiable. Thus for $x\ge 0$
 \bdm
 \begin{split}
  e^{\mu|x|}|f(x)|^2
  &
   = e^{\mu|x|}\bigg|
   \int_x^\infty \frac{d}{ds} |f(s)|^2\, ds\bigg| \\
  &
   \le 2\int_0^\infty e^{\mu|s|}|f(s)||f'(s)|\, ds
   \le 2\|e^{\mu|x|}f\|\|f'\| \le C<\infty
 \end{split}
 \edm
for small enough $\mu>0$. Similarly, $e^{\mu|x|}|f(x)|^2 \le C$ for small enough $\mu$ and $x\le 0$. Analogously, one gets the pointwise exponential decay of $\hatt{f}$. This proves Theorem \ref{thm:exp-decay}.

\vspace{5mm}
\appendix
\setcounter{section}{0}
\renewcommand{\thesection}{\Alph{section}}
\renewcommand{\theequation}{\thesection.\arabic{equation}}
\renewcommand{\thethm}{\thesection.\arabic{thm}}


\section{Proof of Proposition~\ref{prop:mult-lin-bounds}} \label{sec:multilinear-proof}
\setcounter{thm}{0}
\setcounter{equation}{0}
Here we finish our argument by proving Proposition~\ref{prop:mult-lin-bounds}.  Given the bounds from Lemma \ref{lem:intbyparts1} 
and Corollary \ref{cor:small-D}, the argument is formally similar to the multi-linear bounds of
Propositions~2.5 and 2.6 in \cite{EHL}, but requires
more care.  
Many of these difficulties are explained in
Remark~\ref{rem:difficulty}.  Here we remind the interested
reader that
this need for more care arises primarily from
the extention to more general dispersion profiles. 
In particular, we note that the oscillatory integrals one 
needs to 
bound take the form of
$$
	\int e^{-i\frac{a}{4D(s)}}\frac{1}{|D(s)|}\psi(s)\, ds.
$$
In
the model case \eqref{eq:d0-simplest} when
$d_0 = \id_{[0,1)} - \id_{[1,2)}$, one has that
$D(s)=s$ and $\psi=\id_{[0,1]}$.  In allowing for
more general dispersion profiles,
we lose these simplifications and require a somewhat more delicate
use of  the bounds in Section~\ref{sec:Q}. 

First note the bounds
 \beq\label{eq:R-loc}
  \begin{split}
    \int\limits_{\{|\eta_1-\eta_2|\le R\}} \prod_{n=1}^4 |h_n(\eta_n)|\delta(\eta_1-\eta_2+\eta_3-\eta_4)\, d\eta
    &\le
    2R \prod_{n=1}^4 \|h_n\| , \\
    \int\limits_{\{|\eta_2-\eta_3|\le R\}} \prod_{n=1}^4 |h_n(\eta_n)|\delta(\eta_1-\eta_2+\eta_3-\eta_4)\, d\eta
    &\le
    2R \prod_{n=1}^4 \|h_n\| ,
  \end{split}
 \eeq
for all $R\ge 0$. For notational convenience we will abbreviate the measure 
$\delta(\eta_1-\eta_2+\eta_3-\eta_4)\, d\eta$ by $d\rho$. 
To prove \eqref{eq:R-loc} simply note that using Cauchy-Schwarz inequality and the properties of the Dirac measure
one sees, for example,
 \bdm
  \begin{split}
    \int\limits_{\{|\eta_1-\eta_2|\le R\}}
      \prod_{n=1}^4  |h_n(\eta_n)
    \, d\rho 
    \le \,& 
    \Big(
      \int\limits_{\{|\eta_1-\eta_2|\le R\}}
         |h_1(\eta_1)|^2 |h_3(\eta_3)|^2 \, d\rho
    \Big)^{1/2} \\
     &\times \Big(
      \int\limits_{\{|\eta_1-\eta_2|\le R\}}
         |h_2(\eta_1)|^2 |h_4(\eta_3)|^2 \, d\rho
    \Big)^{1/2}
     =
    2R \prod_{n=1}^4 \|h_n\| .
  \end{split}
 \edm
Now we show the bound \eqref{eq:mult-lin} for $K^1$. We split the $s$-integral in the definition \eqref{eq:def-K1} of $K^1$ into a region $\{s\in[0,L]|\, |D(s)|\lesssim |a(\eta)|\}$, where oscillations are important, and a region $\{s\in[0,L]|\, |D(s)|\gtrsim |a(\eta)|\}$, where they are not. Thus
 \begin{align}
  K^1(h_1,h_2,h_3,h_4)
  &\le
    \frac{1}{L} \int\limits_{\R^4}
      \Big| \int\limits_{ \{ |D(s)|\lesssim |a(\eta)|\}}
              e^{\frac{ia(\eta)}{4D(s)}} \frac{ds}{|D(s)|}
      \Big|
     \prod_{n=1}^4 |h_n(\eta_n)|\, d\rho \nonumber \\
  &  +
    \frac{1}{L} \int\limits_{\R^4}
      \Big| \int\limits_{\{ |D(s)|\gtrsim |a(\eta)|\}}
              e^{\frac{ia(\eta)}{4D(s)}} \frac{ds}{|D(s)|}
      \Big|
     \prod_{n=1}^4 |h_n(\eta_n)|\, d\rho  \nonumber \\
  & \lesssim
     \int\limits_{\R^4}
      \frac{1}{\max(\abs{a(\eta)},1)}
     \prod_{n=1}^4 |h_n(\eta_n)\, d\rho \label{eq:K1-part1}\\
  &  +
     \iint\limits_{\substack{(s,\eta)\in [0,L]\times\R^4 \\ |D(s)|\gtrsim |a(\eta)|}}
               \frac{1}{|D(s)|}
     \prod_{n=1}^4 |h_n(\eta_n) \, ds d\rho  \label{eq:K1-part2}
 \end{align}
where we also used the bound \eqref{eq:small-D} for the first part and the triangle inequality
for the second.
Note that $a(\eta)= 2(\eta_1-\eta_2)(\eta_2-\eta_3)$ on the set $\eta_1-\eta_2+\eta_3-\eta_4=0$. Thus,
 \bdm
 \begin{split}
   \{ |D(s)|\gtrsim |a(\eta)|\}
   &=
     \{ |\eta_1-\eta_2||\eta_2-\eta_3|\lesssim |D(s)|\} \\
   &\subset
     \{ |\eta_1-\eta_2|\lesssim \sqrt{|D(s)|}
        \text{ or }
        |\eta_2-\eta_3|\lesssim \sqrt{|D(s)|}\} .
 \end{split}
 \edm
Hence using Fubini-Tonelli and \eqref{eq:R-loc} with $R=\sqrt{|D(s)|}$ for the $d\rho=\delta(\eta_1-\eta_2+\eta_3-\eta_4)d\eta$ integration, one sees
 \begin{align}
   \eqref{eq:K1-part2}
   & \le
     \iint\limits_{\substack{|\eta_1-\eta_2|\lesssim \sqrt{|D(s)|} \\
     \text{or }|\eta_2-\eta_3|\lesssim \sqrt{|D(s)|}}}
               \frac{1}{|D(s)|}
     \prod_{n=1}^4 |h_n(\eta_n) \, d\rho ds   \nonumber \\
   &\lesssim
     \int_0^L \frac{ds}{\sqrt{|D(s)|}}
     \prod_{n=1}^4 \|h_n\|
     \lesssim
     \prod_{n=1}^4 \|h_n\| ,
 \end{align}
where we also used Corollary \ref{cor:sublevel}.

To bound \eqref{eq:K1-part1}, we split the $\eta$-integration into the parts
$\{|\eta_1-\eta_2|\le 1 \text{ or }|\eta_2-\eta_3|\le 1\}$ and
$\{|\eta_1-\eta_2|\ge 1 \text{ and }|\eta_2-\eta_3|\ge 1\}$. Thus
 \begin{align}
   \eqref{eq:K1-part1}
   \lesssim
    &
      \int\limits_{\substack{|\eta_1-\eta_2|\le 1 \\ \text{or }|\eta_2-\eta_3|\le 1}}
      \prod_{n=1}^4 |h_n(\eta_n) \, d\rho  \label{eq:K1-part1-1} \\
    &
      + \int\limits_{\substack{|\eta_1-\eta_2|\ge 1 \\ \text{and }|\eta_2-\eta_3|\ge 1}}
      \frac{1}{|\eta_1-\eta_2||\eta_2-\eta_3|}
     \prod_{n=1}^4 |h_n(\eta_n)| \, d\rho \label{eq:K1-part1-2}
 \end{align}
The bound \eqref{eq:R-loc} with $R=1$ immediately shows
$  \eqref{eq:K1-part1-1}
   \lesssim \prod_{n=1}^4 \|h_n\|.$
On the other hand, the Cauchy-Schwarz inequality yields
 \begin{align}
   \eqref{eq:K1-part1-2}
   \le \,
    &
     \Big(
       \int\limits_{\substack{|\eta_1-\eta_2|\ge 1 \\ \text{and }|\eta_2-\eta_3|\ge 1}}
         \frac{1}{|\eta_1-\eta_2|^2|\eta_2-\eta_3|^2}
       |h_1(\eta_1)|^2\, d\rho
     \Big)^{1/2}
      \Big(
        \int\limits_{\R^4}
        |h_2(\eta_2)|^2|h_2(\eta_3)|^2|h_2(\eta_4)|^2
        \, d\rho
      \Big)^{1/2} \nonumber \\
    &
      = 2 \prod_{n=1}^4\|h_n\|
 \end{align}
where we also used that the Dirac measure kills the $\eta_4$, respectively, $\eta_1$ integration.
Thus
$\label{eq:K1-part1-bound}
    \eqref{eq:K1-part1}\lesssim \prod_{n=1}^4\|h_n\|$
and so \eqref{eq:mult-lin} holds for $ j=1$.

To bound $K^2$, simply note that by Lemma \ref{lem:intbyparts1} we have
  \beq
    \begin{split}
      K^2(h_1,h_2,h_3.h_4)
      &
       \lesssim
        \int\limits_{\R^4}
          \frac{1}{\max(\abs{a(\eta)},1)}
          \prod_{n=1}^4 |h_n(\eta_n)|\, d\rho 
       = \eqref{eq:K1-part1}
         \lesssim
         \prod_{n=1}^4\|h_n\|
    \end{split}
  \eeq
using, in addition, \eqref{eq:K1-part1-bound}. This proves \eqref{eq:mult-lin}.
\smallskip

In the proof of \eqref{eq:mult-lin-loc} we distinguish between the case $l-m$ odd and the case $l-m$ even.  First some more notation. For a (measurable) set $A\subset \R^4$ let
 \begin{align}
  I_1(A)
   &
    :=
     \int\limits_{A}
      \frac{1}{\max(\abs{a(\eta)},1)}
     \prod_{n=1}^4 |h_n(\eta_n)|\delta(\eta_1-\eta_2+\eta_3-\eta_4)\, d\eta, \label{eq:I1} \\
  I_2(A)
   &
    :=
     \iint\limits_{\substack{(s,\eta)\in [0,L]\times A \\ |D(s)|\gtrsim |a(\eta)|}}
               \frac{1}{|D(s)|}
     \prod_{n=1}^4 |h_n(\eta_n)|\delta(\eta_1-\eta_2+\eta_3-\eta_4)\, ds d\eta . \label{eq:I2}
\end{align}
Furthermore, let $A^\tau_{l,m}:=\{ \eta\in\R^4|\, |\eta_l-\eta_m|\ge\tau \}$.

The case $l-m$ odd: First note that by symmetry it is enough to consider the case $(l,m)= (1,2)$. So assume that $\tau=\dist(\supp(h_1), \supp(h_2))>0$. As before, splitting the $s$-integral in the definition \eqref{eq:def-K1} of $K^1$ into the region $\{s\in[0,L]|\, |D(s)|\lesssim |a(\eta)|\}$, where oscillations are important, and a region $\{s\in[0,L]|\, |D(s)|\gtrsim |a(\eta)|\}$, where they are not, and using the bound from Corollary \ref{cor:small-D} on the former set we have
 \begin{align}
  K^1(h_1,h_2,h_3,h_4)
  & \lesssim
     \int\limits_{\R^4}
      \frac{1}{\max(\abs{a(\eta)},1)}
     \prod_{n=1}^4 |h_n(\eta_n)|\, d\rho
     +\iint\limits_{\substack{(s,\eta)\in [0,L]\times\R^4 \\ |D(s)|\gtrsim |a(\eta)|}}
               \frac{1}{|D(s)|}
     \prod_{n=1}^4 |h_n(\eta_n)|\, d\rho ds \nonumber\\
  &
   =   I_1(A^\tau_{1,2}) + I_2(A^\tau_{1,2})
 \end{align}
since $h_1(\eta_1)h_2(\eta_2)=0$ if $|\eta_1-\eta_2|<\tau$. In the integral defining $I_2(A^\tau_{1,2})$ we have $|D(s)|\gtrsim |a(\eta)|=2|\eta_1-\eta_2||\eta_2-\eta_3|\ge 2\tau |\eta_2-\eta_3|$. Hence, using \eqref{eq:R-loc} with $R=|D(s)|/\tau$ for the $\eta$-integration, we have
 \beq\label{eq:I2bound}
  I_2(A^\tau_{1,2})
   \lesssim
    \iint\limits_{\substack{(s,\eta)\in [0,L]\times \R^4 \\ |\eta_2-\eta_3|\lesssim |D(s)|/\tau}}
               \frac{1}{|D(s)|}
     \prod_{n=1}^4 |h_n(\eta_n)|\delta(\eta_1-\eta_2+\eta_3-\eta_4)\, d\eta ds
    \lesssim
     \frac{1}{\tau}\prod_{n=1}^4\|h_n\| .
 \eeq
To estimate $I_1(A^\tau_{1,2})$ we split $A^\tau_{1,2}= B_1\cup B_2$ with
 \begin{align*}
  B_1
  &
   := \{\eta\in\R^4:\, |\eta_1-\eta_2|\ge \tau \text{ and } |\eta_2-\eta_3| \ge 1\} \\
  B_2
  &
   := \{\eta\in\R^4:\, |\eta_1-\eta_2|\ge \tau \text{ and } |\eta_2-\eta_3| \le 1\} .
 \end{align*}
Obviously, $I_1(A^\tau_{1,2}) = I_1(B_1) + I_1(B_2)$
and, similarly as for the bound of \eqref{eq:K1-part1-2}, the Cauchy--Schwarz inequality yields
 \beq
 \begin{split}
  I_1(B_1)
    &
     =
     \int\limits_{B_1}
       \frac{1}{\max(\abs{a(\eta)},1)}
       \prod_{n=1}^4 |h_n(\eta_n)|\, d\rho\\
    &
     \lesssim
     \Big(
       \int\limits_{\substack{|\eta_1-\eta_2|\ge \tau \\ \text{and }|\eta_2-\eta_3|\ge 1}}
         \frac{1}{|\eta_1-\eta_2|^2|\eta_2-\eta_3|^2}
       |h_1(\eta_1)|^2\, d\rho
     \Big)^{1/2} 
      \Big(
        \int\limits_{\R^4}
        |h_2(\eta_2)|^2|h_2(\eta_3)|^2|h_2(\eta_4)|^2\, d\rho
      \Big)^{1/2} \nonumber\\
     &
      \lesssim \tau^{-1/2}
\prod_{n=1}^4\|h_n\|
  \end{split}
 \eeq
Since $\max(|a|,1) \ge |a|^{1/2}$, we also have
 \beq
 \begin{split}
  I_1(B_2)
   &
    \lesssim
    \int\limits_{B_2}
      \frac{1}{|\eta_1-\eta_2|^{1/2}|\eta_2-\eta_3|^{1/2}}
      \prod_{n=1}^4 |h_n(\eta_n)|\, d\rho \\
   &
    \le
    \tau^{-1/2}
    \int\limits_{B_2}
     \frac{1}{|\eta_2-\eta_3|^{1/2}}
     \prod_{n=1}^4 |h_n(\eta_n)|\, d\rho .
 \end{split}
 \eeq
Using Cauchy--Schwarz with respect to the measure $\delta(\eta_1-\eta_2+\eta_3-\eta_4)\,d\eta_1d\eta_4$, one sees
 \bdm
  \begin{split}
   \int\limits_{\R^2} & |h_1(\eta_1)||h_4(\eta_4)|\delta(\eta_1-\eta_2+\eta_3-\eta_4) \,d\eta_1 d\eta_4 \\
   &
   \le
    \Big(
      \int\limits_{\R^2} |h_1(\eta_1)|^2\delta(\eta_1-\eta_2+\eta_3-\eta_4) \,d\eta_1 d\eta_4
    \Big)^{1/2}
    \Big(
      \int\limits_{\R^2} |h_4(\eta_4)|^2\delta(\eta_1-\eta_2+\eta_3-\eta_4) \,d\eta_1 d\eta_4
    \Big)^{1/2} \\
   &= \|h_1\| \|h_4\|.
  \end{split}
 \edm
uniformly in $(\eta_2,\eta_3)\in\R^2$.
Hence
 \bdm
 \begin{split}
  \int\limits_{B_2} &
     \frac{1}{|\eta_2-\eta_3|^{1/2}}
     \prod_{n=1}^4 |h_n(\eta_n)
    \, d\rho \\
  & =
    \int\limits_{|\eta_2-\eta_3|\le 1}  \frac{|h_2(\eta_2)||h_3(\eta_3)|}{|\eta_2-\eta_3|^{1/2}}
      \Big(
        \int\limits_{\R^2} |h_1(\eta_1)||h_4(\eta_4)|\delta(\eta_1-\eta_2+\eta_3-\eta_4) \,d\eta_1 d\eta_4
      \Big)
    \,d\eta_2d\eta_3
      \\
  &
   \le
    \|h_1\| \|h_4\|
    \int\limits_{|\eta_2-\eta_3|\le 1}
     \frac{|h_2(\eta_2)||h_3(\eta_3)|}{|\eta_2-\eta_3|^{1/2}}
    \,d\eta_2d\eta_3  \\
  &
   \le
    \|h_1\| \|h_4\|
    \Big(
      \int\limits_{|\eta_2-\eta_3|\le 1} d\eta_2d\eta_3\frac{|h_2(\eta_2)|^2}{|\eta_2-\eta_3|^{1/2}}
    \Big)^{1/2}
    \Big(
      \int\limits_{|\eta_2-\eta_3|\le 1} d\eta_2d\eta_3\frac{|h_3(\eta_3)|^2}{|\eta_2-\eta_3|^{1/2}}
    \Big)^{1/2} \\
  &
   \lesssim
    \|h_1\| \|h_2\| \|h_3\| \|h_4\|
 \end{split}
 \edm
where we also used the Cauchy--Schwarz inequality with respect to the measure  $|\eta_2-\eta_3|^{1/2}\,d\eta_2d\eta_3$ in the second inequality. Thus
 \beq\label{eq:I1bound}
   I_2\lesssim \tau^{-1/2} \prod_{n=1}^4 \|h_n\|.
 \eeq
The bounds \eqref{eq:I2bound} and \eqref{eq:I1bound} together prove the bound for
$K^1$. The bound for $K^2$ follows immediately from \eqref{eq:I1bound}, since
 \beq
  K^2(h_1,h_2,h_3,h_4)\lesssim I_1(A^\tau_{1,2}) .
 \eeq
\smallskip

Now we prove the case that $l-m$ is even. By symmetry, it is enough to consider the case  $(l,m)=(1,3)$. So assume that $\tau=\dist(\supp(h_1),\supp(h_3))>0$. In this case we have
$K^1(h_1,h_2,h_3,h_4)\lesssim I_1(A^\tau_{1,3})+ I_2(A^\tau_{1,3})$ and
$K^2(h_1,h_2,h_3,h_4)\lesssim I_1(A^\tau_{1,3})$.

The triangle inequality yields $A^\tau_{1,3}\subset A^{\tau/2}_{1,2}\cup A^{\tau/2}_{2,3}$ and the bounds
\eqref{eq:I2bound} and \eqref{eq:I1bound} proven in the case $l-m$ odd now yield
 \begin{align*}
  &I_1(A^\tau_{1,3})\le I_1(A^{\tau/2}_{1,2})+I_1(A^{\tau/2}_{2,3})
  \lesssim \tau^{-1/2}\prod_{n=1}^4\|h_n\|,\\
  &I_2(A^\tau_{1,3})\le I_1(A^{\tau/2}_{1,2})+I_1(A^{\tau/2}_{2,3})
  \lesssim \tau^{-1/2}\prod_{n=1}^4\|h_n\|.
 \end{align*}
This finishes the proof of Proposition \ref{prop:mult-lin-bounds}.

\noindent
\textbf{Acknowledgements: } We thank Vadim Zharnitsky for discussions and piquing our interest in the dispersion management technique. Financial support by the Alfried Krupp von Bohlen und Halbach Foundation and the National Science Foundation under grant DMS-0803120 (D.~H.) and grant DMS-0900865, the American Mathematical Society and the Simons Foundation (W.~R.~G.) is gratefully acknowledged.


\renewcommand{\thesection}{\arabic{chapter}.\arabic{section}}
\renewcommand{\theequation}{\arabic{chapter}.\arabic{section}.\arabic{equation}}
\renewcommand{\thethm}{\arabic{chapter}.\arabic{section}.\arabic{thm}}

\end{document}